\documentclass[a4paper,reqno]{article}

\usepackage[english]{babel}
\usepackage[utf8x]{inputenc}
\usepackage{enumerate,amsmath,amssymb,amstext,amsthm}
\usepackage{graphicx}
\usepackage[colorinlistoftodos]{todonotes}
\usepackage{authblk}
\usepackage{mathtools}
\usepackage{hyperref}
\usepackage{afterpage}

\newtheorem{theorem}{Theorem}[section]
\newtheorem{lemma}[theorem]{Lemma}
\newtheorem{corollary}[theorem]{Corollary}

\theoremstyle{definition}
\newtheorem{defn}[theorem]{Definition}

\newcommand{\R}{\mathbb{R}}
\newcommand{\Z}{\mathbb{Z}}
\DeclareMathOperator{\spn}{span}

\newcommand{\one}{\ensuremath{{\bf 1}}}
\newcommand{\zero}{\ensuremath{{\bf 0}}}

\def\be{\begin{equation}}
\def\ee{\end{equation}}

\newcommand{\chiv}{\chi_{v}}

\newcommand{\sym}{\mathcal{S}}
\newcommand{\hnk}{H_{n,k}}

\newcommand{\qeds}{\qed\vspace{.2cm}}
\newcommand{\trans}{{\sf T}}
\DeclareMathOperator{\tr}{Tr}

\DeclareMathOperator{\im}{Im}
\DeclareMathOperator{\gram}{Gram}

\title{Graph Homomorphisms via Vector Colorings}
\author[1]{Chris Godsil}
\author[2]{David E.~Roberson}
\author[3]{Brendan Rooney}
\author[4]{Robert \v{S}\'{a}mal}
\author[5,6]{Antonios Varvitsiotis}
\affil[1]{\small{Department of Combinatorics \& Optimization, University of Waterloo}}
\affil[2]{Department of Computer Science, University College London}
\affil[3]{Department of Mathematical Sciences, KAIST}
\affil[4]{Computer Science Institute, Charles University}
\affil[5]{Centre for Quantum Technologies, National University of Singapore}
\affil[6]{School of Physical and Mathematical Sciences, Nanyang Technological University}

\begin{document}
\maketitle
\begin{abstract}
In this paper we study the existence of homomorphisms  $G\to H$ using  semidefinite programming.  Specifically, we use the vector chromatic number of a graph, defined  as the smallest real number $t \ge 2$ for which  there exists an assignment of unit vectors $i\mapsto p_i$  to its vertices such that $\langle p_i, p_j\rangle\le -1/(t-1),$ when~$i\sim  j$. Our approach  allows to reprove, without using the Erd{\H o}s-Ko-Rado Theorem, that for  {$n>2r$} the Kneser graph $K_{n:r}$ and the  $q$-Kneser graph $qK_{n:r}$ are cores,  and furthermore,   that  for    $n/r = n'/r'$  there exists a homomorphism  $K_{n:r}\to K_{n':r'}$ if and only if $n$ divides $n'$. In terms of new applications,  we show  that the even-weight component of the distance $k$-graph of the $n$-cube  $\hnk$   is  a core and also, that non-bipartite Taylor graphs are cores. Additionally,  we give a necessary and sufficient condition for the existence of homomorphisms~$\hnk\to H_{n',k'}$ when   $n/k = n'/k'$. Lastly, we  show that if a 2-walk-regular graph (which is non-bipartite and not complete multipartite) has a unique  optimal vector coloring, it is a core. Based on this sufficient condition  we conducted a computational study  on  Ted Spence's list  of strongly regular~graphs \cite{Spence}  and found that at least 84\% are cores.
%\tableofcontents

 %and moreover, it  remains a core  after adding all edges between vertices with   Hamming distance greater than $k$. 

%We  recover the well-known result that if $n>2r $ and  $n/r = n'/r'$, there exists a homomorphism  $K_{n:r}\to K_{n':r'}$ if and only if $n$ divides $n'$. Furthermore, we show that  {the necessary condition of} this  theorem  generalizes to $qK_{n:r},$ where  all numbers are replaced by their $q$-analogues.      Lastly, we give a necessary and sufficient condition for the existence of homomorphisms~$\hnk\to H_{n',k'}$ when  $\chiv(\hnk) = \chiv(H_{n',k'})$.\footnote{A subset of the results appearing in this paper was published as an extended abstract in the proceedings of the 8th European Conference on Combinatorics, Graph Theory and Applications, EuroComb 2015~\cite{eurocomb}.}

% and $r'$ are integer multiples of $n$ and $r$ respectively.
 
   %This allows us to recover the known result that the Kneser graph $K_{n:r}$ and the $q$-Kneser graph $qK_{n:r}$ are cores for $n\ge 2r+1$. Our  proof is simple and  does not rely on the use of an Erd\H{o}s-Ko-Rado type result as do existing proofs. Furthermore,  we also show that a  new family of  graphs from the binary Hamming scheme are cores, that was not known before.

\end{abstract}
%\tableofcontents

%%%%%%%%%%%%%%%%%%%%%%%%%%%%%%%%%%%%%%%%%%%%
%%%%%%%%%%%%%%%%%%%%%%%%%%%%%%%%%%%%%%%%%%%%
\section{Introduction}\label{sec:intro}
%%%%%%%%%%%%%%%%%%%%%%%%%%%%%%%%%%%%%%%%%%%%
%%%%%%%%%%%%%%%%%%%%%%%%%%%%%%%%%%%%%%%%%%%%

 A \emph{homomorphism} from a graph $G$ to a graph $H$ is an adjacency preserving map from $V(G)$ to $V(H)$. Formally, a function $\varphi:V(G)\rightarrow V(H)$ is a homomorphism from $G$ to $H$ if $\varphi(u)$ and $\varphi(v)$ are adjacent in $H$  whenever $u$ and $ v$ are adjacent in $G$. If there exists  a homomorphism from $G$ to $H$ we write $G\rightarrow H$. 

Many well-known graph parameters can be defined through graph homomorphisms. One example is the \emph{chromatic number} of a graph $G$, denoted by $\chi(G)$, which is defined as is the least  number of colors required to color the vertices of $G$ such that no two adjacent vertices receive the same color. Equivalently, $\chi(G)$ is the minimum integer $m$ so that $G\rightarrow K_m$, where $K_m$ is the complete graph on $m$ vertices. Other examples include the clique number, the fractional chromatic number, and the circular chromatic number. 
%and the odd girth. 
The interested reader is referred to~\cite{HahnTardif, HN04} for an extensive survey of graph homomorphisms.  

In this work we study the existence of  homomorphisms from a  graph $G$ to a graph~$H$. 
This  problem is important as  many important graph-theoretic questions can be phrased as deciding the existence of a homomorphism between two  graphs. %s there a homomorphism from a  graph $G$ to a graph~$H$? This problem is our focus in this paper.
Nevertheless, it is known that  for  a non-bipartite graph $H$, deciding whether a graph has a homomorphism to $H$  is NP-hard~\cite{HN}. % As a consequence, we study th our attention to specific graph classes.  

% This implies that any general result characterizing the existence of homomorphisms will be necessarily small in scope.

In this paper we study the existence of homorphisms  $G\to H$  using semidefinite programming, and more specifically, using vector colorings. For $d\geq 1$ and $t\geq 2$,  let $ \mathcal{S}_t^d$  be  the infinite graph whose vertices are the unit vectors in $\mathbb{R}^d$, where two unit vectors are adjacent if and only if their inner product is at most $-1/(t-1)$. A homomorphism from $G$ to $\mathcal{S}_t^d$ is called a \emph{vector $t$-coloring} of $G$. Equivalently, a vector $t$-coloring of $G$ is an assignment $i\mapsto p_i$  of  unit vectors in $\mathbb{R}^d$ to the vertices of $G$ such that 
\begin{equation}\label{condition}
\langle p_i,p_j\rangle \le \frac{-1}{t-1}\quad\text{whenever}\quad i\sim j,
\end{equation}
where $\langle \cdot,\cdot\rangle$ denotes the standard  inner product in $\R^d$. Note that we will often assume that the vertex set of a graph is $[n] = \{1, \ldots, n\}$ unless otherwise specified.

The \emph{vector chromatic number} of $G$ is the smallest $t\geq 2$ for which $G\rightarrow\mathcal{S}_t^d$ (for some integer $d\ge 1$) and is denoted $\chi_v(G)$. By convention, the vector chromatic number of the empty graph is equal to one. We call a vector $t$-coloring of $G$ \emph{optimal} if $t=\chi_v(G)$. Note that  without loss of generality  we can always set $d=|V(G)|$, as the space spanned by (the images of) the vertices of $G$ has dimension at most $|V(G)|$. A vector $t$-coloring is \emph{strict} if every inequality in~\eqref{condition} is met with equality. The \emph{strict vector chromatic number} of $G$, denoted $\chi_{sv}(G)$, is the smallest $t\geq 2$ for which $G$ has a strict vector $t$-coloring.

Both $\chi_v(G)$ and $\chi_{sv}(G)$ were originally introduced by Karger et al.~\cite{KMS} as relaxations of $\chi(G)$. These parameters satisfy the relation $\chi_v(G)\le\chi_{sv}(G)\le~\chi(G)$. Karger et al.~noted that $\chi_{sv}(G)$ is the Lov{\'a}sz theta function of the complement of $G$~\cite{KMS}. Furthermore, $\chi_v(G)=\vartheta'(\overline{G})$ where $\vartheta'$ is a variant of the Lov{\'a}sz theta introduced in~\cite{schrijver} and~\cite{McRR}.

We study the existence of homorphisms  $G\to H$ when  the  graphs $G$ and $ H$ have  the same vector chromatic number, i.e.,  $\chi_v(H)=\chi_v(G)$. The high-level idea is the following. If $\varphi_1$ is an optimal vector coloring of $H$  and $\varphi_2$  a homomorphism from $G$ to $H$,  the map  $\varphi_1\circ\varphi_2$ is  an optimal  vector coloring of $G$. As a consequence, properties  of  optimal vector colorings of $G$   translate to properties of homomorphisms $G\to H$. As an example, if all optimal vector colorings of $G$ are injective, any homomorphism $G\to H$ will also be injective. 

Since  the vector chromatic number  of a graph is given by a semidefinite program,  an optimal vector coloring  can be identified to arbitrary precision in polynomial time. Nevertheless,  finding the set of  {\em all} optimal vector colorings is in general a hard problem.  For this reason, in this paper we further restrict our attention to graphs  that are  \emph{uniquely vector colorable}~(UVC), i.e.,  any two optimal vector colorings   differ only by an orthogonal transformation.   

Formally, a graph $G$ is called \emph{uniquely (strict) vector colorable} if for any two 
{optimal} (strict) vector colorings $i\mapsto p_i\in \R^d$ and $i\mapsto q_i\in \R^{d'}$ the corresponding Gram matrices coincide, i.e., 
\be\label{equalgram}
\gram(p_1, \ldots, p_n) = \gram(q_1, \ldots, q_n).\
\ee
We  say that $i\mapsto p_i$ is the {\em unique optimal vector coloring} of $G$ if  for any other optimal vector coloring 
$i\mapsto q_i$,  Equation \eqref{equalgram} holds. Furthermore, we say that two vector colorings  $i\mapsto p_i$ and $i\mapsto q_i$ are {\em congruent} if they  satisfy~\eqref{equalgram}.

% of $\mathbb{R}^d$. 

Although deciding whether a graph is UVC is hard, there exist sufficient conditions for showing that a graph is UVC. For example, such conditions were developed in \cite{UVC1} where  it was shown that  the \emph{Kneser graphs} $K_{n:r}$ and their $q$-analogs, the \emph{q-Kneser graphs} $qK_{n:r}$ are UVC. These graphs have nontrivial structure: the vertex set of $K_{n:r}$ consists of the $r$-subsets of $[n]$, with disjoint subsets being adjacent. Similarly, the vertices of $qK_{n:r}$ are the $r$-dimensional subspaces of $\mathbb{F}_q^n$, two being adjacent if they intersect in the trivial subspace. Furthermore, UVC graphs are interesting in their own right. They were first introduced in~\cite{pak} to construct tractable instances of the graph realization problem. In the same work UVC graphs were used to construct uniquely colorable graphs. UVC graphs are also closely related to the notion of universal completability (equivalently, the universal rigidity of apex graphs). This in turn is relevant to the low-rank matrix completion problem~\cite{LV}.

%These two assumptions will allow us to deduce properties of homomorphisms $\varphi_2$.

% In order to establish the existence of homomorphisms from $G$ to $H$, we will work with graphs $G$ whose vector colorings have particular structure.
%

%\subsection{Contributions and related work}

\subsection{Summary of  results and paper organization} 
%Our approach relies crucially on the  ability to 
%This is the problem we consider in the first part of this work. 
% In Section \ref{} we   determine  another family of UVC graphs $H_{n,k}$ (Theorem~\ref{thm:hnk_is_uvc}). These graphs are constructed from the $k$-distance graphs of the $n$-cube.

%
%
%In Section \ref{subsec:1walkreg} we focus on 1-walk-regular graphs (see Definition~\ref{def:1walkreg}). The class of 1-walk-regular graphs is central to this work as we know an optimal vector coloring for all such graphs which we call the {\em canonical vector coloring} ({see} Definition \ref{def:canonicalvc}). Furthermore, there  exists a necessary and sufficient condition for showing  they are UVC  (Theorem \ref{thm:1walkreg}). Both of  these results were shown in~\cite{UVC1}.
%
%
%
%Standard examples of 1-walk-regular graphs are the \emph{Kneser graphs} $K_{n:r}$ and their $q$-analogs, the \emph{q-Kneser graphs} $qK_{n:r}$. These graphs were shown to be UVC in~\cite{UVC1}, and for completeness we explicitly give their vector colorings in Section~\ref{sec:examplesUVC}.
%In Section~\ref{sec:nonstrict} we show how to construct UVC graphs whose unique vector colorings are not strict (Theorem~\ref{thm:strict2nonstrict}). We note that this is not possible using the sufficient condition from~\cite{UVC1} (Theorem~\ref{thm:1walkreg}). Using this construction we obtain graphs whose vector chromatic number is smaller than their strict vector chromatic number. Few examples of such graphs are known (e.g., see~\cite{schrijver}).

\paragraph{Graph endomorphisms.} In Section \ref{sec:cores}   we study the existence of  \emph{endomorphisms} of a graph $G$, i.e.,  homomorphisms from  $G$ to itself. 
%In this setting we have $H=G$, so $\chi_v(G)=\chi_v(H)$ trivially. 
Our goal is to find sufficient conditions to show that $G$ does not admit any  endomorphisms to a proper subgraph.  Graphs that have  this property are known as \emph{cores}.

For an arbitrary graph $G$, the \emph{core of $G$} is the vertex minimal subgraph to which $G$ admits a homomorphism. Every graph has a unique core, and the core of $G$ is itself a core. Moreover, $G$ and $H$ have the same core if and only if they are \emph{homomorphically equivalent}, i.e., $G\rightarrow H$ and $H\rightarrow G$. Cores are the unique minimal elements of these homomorphic equivalence classes. In this sense, the core of a graph is the smallest graph retaining all  its  homomorphic~information. %Although several graphs are known to be cores, there is no systematic way to test whether a given graph is a core.

It is known that deciding whether a graph is a core is a co-NP-complete problem~\cite{NPcore}.
In Section~\ref{sec:suffcores} we show that if $G$ is UVC, and its unique optimal vector coloring is injective on the neighborhood of each vertex, then $G$ is a core.
%on its vertex neighborhoods,  %(cf. Theorem~\ref{thm:vect2core}).

 To illustrate the usefulness of this  sufficient condition, in Section~\ref{sec:knesercores} we show   that for $n\geq 2r+1$, both the Kneser graph $K_{n:r}$ and the $q$-Kneser graph   $qK_{n:r}$    are cores. Although this    is well-known~\cite{HahnTardif},  our proof avoids invoking the Erd{\H o}s-Ko-Rado Theorem,  used to describe the structure of the maximum independent sets of these graphs, and it also avoids using the No Homomorphism~Lemma~\cite{GR}.
 
 In terms of new applications, we  show that a family of Hamming distance graphs constructed from the $k$-distance graphs of the $n$-cube are cores.  These graphs, denoted $H_{n,k}$, have the even weight binary strings of length $n$ as their vertices, two being adjacent if they differ in precisely $k$ positions. In Section~\ref{sec:hamming}, we show that these graphs are UVC for even $k \in [n/2 +1,n-1]$.
%MAYBEADDBACKAdditionally, we show that when the conditions of Theorem~\ref{thm:vect2core} are met, we can sometimes add edges to a core and have it remain a core (Theorem~\ref{thm:augmentcores}). This is interesting, as the property of being a core usually behaves unpredictably under graph modifications. 
In Section~\ref{subsec:2walkreg} we focus on  2-walk-regular graphs. We show that if a 2-walk-regular graph (that is not bipartite or complete multipartite) is UVC,  it is also a core.  Furthermore, in Section \ref{sec:Taylor} we show that non-bipartite Taylor graphs are~cores.

Finally, in Section~\ref{sec:SRGcomps} we give an  algorithm for testing whether a 2-walk-regular graph is a core. %Our algorithm is based on Lemma~\ref{lem:uvctest}, which characterizes unique vector colorability in terms of the dimension spanned by a set of matrices constructed from the canonical vector coloring of a graph. 
%Our algorithm has polynomial time complexity and appears to be very efficient in practice. 
We apply this  algorithm  to $73816$ strongly regular graphs obtained from Ted Spence's webpage~\cite{Spence}, showing that $62168$ (approx.~$84\%$) of them are UVC,  and therefore cores.

%So out of the $11648$ instances for which our  algorithm failed to return an answer,  $67\%$ of them were not cores. 

%This seems to indicate that not being a core is a very rare property for strongly regular graphs. Also, this shows that our UVC test found a great deal of the cores, but still missed a significant fraction of them. Note, however, that computing clique or chromatic number is NP hard, and the latter is also typically very hard in practice. Thus this approach for showing a strongly regular graph is a core will become untenable long before our algorithm does.

%We use this to perform a computational study on 73816 strongly regular graphs listed on Ted Spence's website that have integral eigenvalues. This gives a computational proof that 84\% of these graphs are UVC and therefore cores. 

%(Theorem~\ref{thm:2walkreg}). %We can formulate an algorithm for testing whether a 2-walk-regular graph is UVC. So this leads to an algorithm for testing whether a 2-walk-regular graph is a core. In Section 4.3 we apply this approach to strongly regular graphs. We summarize a computational study of the 73816 strongly regular graphs listed on Ted Spence's website that have integral eigenvalues. We prove that 84\% of these graphs are UVC, and thus are cores.

\paragraph{Homomorphisms between graphs with $\chi_v(G)=\chi_v(H)$.} In Section~\ref{sec:equalchivec}, we study necessary and sufficient conditions for the existence of homomorphisms from $G$ to $H$ for a pair of graphs satisfying  $\chi_v(G)=\chi_v(H)$. %Our focus is on graphs that lie in the same family.

In Section \ref{sec:kneser}  we focus  on Kneser graphs.  It is an open problem to determine  all possible homomorphisms between Kneser graphs (e.g., see \cite[Problem 11.2]{G95}). On the positive side, using  the Erd{\H o}s-Ko-Rado~Theorem,  Stahl showed in~\cite{stahl}  that if $n/r=n'/r'$,  then $K_{n:r}\rightarrow K_{n':r'}$ if and only if $n'$ is an integer multiple of~$n$. As the condition $n/r=n'/r'$ is equivalent to $\chi_v(K_{n:r})=\chi_v(K_{n':r'})$,  we are able to reprove this result using our approach. %(see Theorem~\ref{thm:kneser_hom}). 

In Section \ref{sec:qknesernesssufficient} we  consider the family of $q$-Kneser graphs. Again, we study the existence of  homomorphisms from $qK_{n:r}$ to $q'K_{n':r'}$ where $\chi_v(qK_{n:r})=\chi_v(q'K_{n':r'})$. Our main result is that, under this assumption, the existence of a homomorphism from $qK_{n:r}$ to $q'K_{n':r'}$ implies that the $q$-binomial coefficient $[n']_q$ is an integer multiple of the $q'$-binomial coefficient $[n]_{q'}$. %(Theorem~\ref{thm:qknesernessecary}). 

Finally, in Section \ref{sec:hamming_homs} we give necessary and sufficient conditions for the existence of homomorphisms  $H_{n,k} \to H_{n',k'}$ when $\chi_v(H_{n,k})=\chi_v(H_{n',k'})$. %(Theorem~\ref{thm:homomor_hnk}).

  \section{Preliminaries}
  
\subsection{Basic definitions and notation}\label{sec:definitions}

Throughout we set $[n]=\{1,\ldots,n\}$. We denote by $e_i$ the $i^\text{th}$ standard basis vector, by $\one$ the all-ones vector and by $\zero$ the all-zeros vector of appropriate size.  All vectors are column  vectors. We denote by $\langle \cdot ,\cdot  \rangle $ the usual inner product between two real vectors. Furthermore, we denote by $\spn(p_1,\ldots,p_n)$ the linear span of the vectors $\{p_i\}_{i=1}^n$.   
The set of  $n\times n$ real symmetric matrices is denoted by $\sym^n$, and the set of matrices in $\sym^n$ with nonnegative eigenvalues, i.e., the real positive semidefinite matrices, is denoted by $\sym^n_+$. Given a matrix $X\in \sym^n$ we denote its kernel/null space by ${\rm Ker} X$  and its image/column space by $\im X$. %The  {\em corank} of a matrix $X\in \sym^n$, denoted $\cor X$, is defined as the dimension of~$\Ker X$. 
The {\em Schur} product of two matrices $X,Y\in \sym^n$, denoted by $X\circ Y$,  is the matrix whose entries are given by $(X\circ Y)_{ij}=X_{ij}Y_{ij}$ for all $i,j\in [n]$. %Lastly, the {\em direct sum} of two matrices $X,Y\in \sym^n$, denoted by $X\oplus Y$,  is given by the  matrix
%\[\begin{pmatrix} X& 0\\
 %0 & Y\end{pmatrix}.\]
 A  matrix $X\in \sym^n$ has real eigenvalues, and we denote the smallest one by~$\lambda_{min}(X)$.   The {\em Gram matrix} of a set of vectors $v_1, \ldots, v_n$, denoted  by $\gram(v_1, \ldots, v_n)$, is the $n\times n$  matrix with $ij$-entry equal to $\langle  v_i, v_j\rangle $. The matrix $\gram(v_1, \ldots, v_n)$ is positive semidefinite and its rank is equal to the dimension of  $  \spn(p_1,\ldots,p_n)$. We denote by ${\rm sum}(X)$ the sum of all entries in $X$  and use that ${\rm sum}(X\circ Y)=~{\rm Tr}(XY^T)$.
 
  %For any graph $G$ we denote by $\overline{G}$ the \emph{complement} of $G$, i.e., the graph obtained by replacing the edges of $G$ with non-edges and vice versa. 
  \subsection{1-walk-regular graphs} 
A graph   $G$ with adjacency matrix $A$ is said to be \emph{1-walk-regular} if for all $k \in \mathbb{N}$, there exist constants $a_k$ and $b_k$ such that
\begin{itemize}
\item[$(i)$] $A^k \circ I = a_k I$;
\item[$(ii)$]  $A^k \circ A = b_k A$.
\end{itemize}
%where $\circ$ denotes the entrywise product of two matrices.  
Equivalently, a graph is 1-walk-regular if for all $k \in \mathbb{N}$, $(i)$ the number of walks of length $k$ starting and ending at a vertex does not depend on the choice of vertex, and $(ii)$ the number of walks of length $k$ between the endpoints of an edge does not depend on the edge.

Note that a 1-walk-regular graph must be regular. Also, any graph which is vertex and edge transitive is easily seen to be 1-walk-regular. More generally, any graph which is a single class of an association scheme is 1-walk-regular. These include distance regular and, more specifically, strongly regular graphs, the latter of which is  the focus of Section~\ref{sec:SRGcomps}.

%  refer to any such framework as \emph{the} least eigenvalue framework of~$G$. 

% 1-walk-regular graphs are of relevance to this work  since there  exists a necessary and sufficient condition for being UVC, in terms of a canonical optimal vector coloring.  
 
 Graphs that are 1-walk-regular are particularly relevant to this work because they have a \emph{canonical vector coloring} and furthermore, there exists  a necessary and sufficient condition for this to be the unique vector coloring of such a graph. We first  give the definition of the canonical vector coloring.

 % in terms of the vectors $\{p_i\}_{i=1}^n$  for showing that  $G$ is~UVC.

\begin{defn}\label{def:canonicalvc} 
Consider a  1-walk-regular graph $G=([n],E)$ and let $d$ be the multiplicity of   the least eigenvalue of its adjacency matrix. Furthermore, let $Q$ %=[q_1,\ldots, q_d]$ 
be an $n\times d$ matrix whose columns form  an orthonormal basis for  the eigenspace of the least eigenvalue of~$G$ and let $p_i\in \R^d$ be the $i$-th row of $Q$. %To each vertex $i\in [n]$ we associate the vector %$q_i\in \R^d$ that corresponds to  the $i$-th row of $Q$.
 The assignment $i\mapsto  \sqrt{\frac{n}{d}}p_i\in~\R^d$ is  a vector coloring of $G$  which we call the  {\em canonical vector~coloring}. 
\end{defn}

Consider  a 1-walk-regular graph $G$ with least eigenvalue $\tau$. Note that the  vectors in the  canonical vector coloring linearly span the ambient space, i.e., ${\rm span}(p_1,\ldots,p_n)=~\R^d$. Also,     the canonical vector coloring  of  $G$ is not uniquely defined  since there are many choices  of orthonormal basis for the least eigenspace. Nevertheless, all canonical vector colorings are congruent and thus indistinguishable for our purposes. Indeed,  for any orthonormal basis %$\{q_i\}_{i=1}^d$ 
of the least eigenspace,  the   Gram matrix  of the corresponding canonical vector coloring  is equal to a scalar multiple of  the orthogonal projector $E_\tau$ onto the least eigenspace of $G$. To see this, let $Q$ be the matrix whose columns are the chosen orthonormal basis vectors, and consider how the matrix $QQ^\trans$ acts on the least eigenspace of $G$ and its orthogonal complement.
Furthermore, it follows  by the definition of a canonical vector coloring that  $\tau p_i = \sum_{j \sim i} p_j$ for all $i\in V(G)$.

Lastly, recall that  the projector $E_\tau$ onto the least eigenspace of a graph  $G$ is a polynomial in the adjacency matrix of $G$. Concretely, we have that $E_\tau=\prod_{\lambda\ne \tau}{1\over \tau-\lambda}(A-\lambda I)$. Thus, if $G$ is   1-walk-regular, the diagonal entries of  $E_\tau$ and the entries of $E_\tau$ that correspond to edges of $G$ are~constant. 
 %This follows directly from the definition above.

We are now ready to give a necessary and sufficient condition for a $1$-walk-regular graph to be UVC.

\begin{theorem}[\cite{UVC1}] \label{thm:1walkreg}
Let $G=([n],E)$ be 1-walk-regular with degree $k$ and let $ i \mapsto p_i\in \R^d$ be  its canonical vector coloring. Then, we have that:
\begin{itemize}
\item[$(i)$]  $\chiv(G)=1-{k\over  \lambda_{{\rm min}}(G)}$ and  $ i \mapsto p_i$ is an optimal  strict vector coloring of $G$.%, where  $k$ is the valency of $G$ and $\tau$ is the least eigenvalue of its adjacency matrix.
\item[$(ii)$] $G$ is uniquely vector   colorable if and only if  for any $R\in \mathcal{S}^d$ we have
\be\label{eq:conicatinfinity}
p_i^\trans R p_j = 0 \ \text{ for all } \ i \simeq j \Longrightarrow R=0.
\ee
where $ i \simeq j$ means that the vertices $i$ and $ j$ are either equal or adjacent. 
\end{itemize} 
\end{theorem} 

We note that  the calculation for the vector chromatic number of a 1-walk-regular graph was first done in  \cite[Lemma 5.2]{sabvshed}.

\section{Graph cores}\label{sec:cores}
%%%%%%%%%%%%%%%%%%%%%%%%%%%%%%%%%%%%%%%%%%%%
%%%%%%%%%%%%%%%%%%%%%%%%%%%%%%%%%%%%%%%%%%%%
%\subsection{The sufficient condition}\label{sec:sufficientcondcores}
\subsection{A sufficient condition for a graph to be a core}\label{sec:suffcores}

A homomorphism $\varphi$ is \emph{locally injective} if it acts injectively on the neighborhood of any vertex, i.e.,~if $\varphi(u) \ne \varphi(v)$ for any two vertices $u$ and $v$ that have a common neighbor. %In other words, a homomorphism is locally injective if it does not identify any two vertices at distance two from each other. 
We recall the  following property of endomorphisms  proved by Ne\v{s}et\v{r}il  which we use  to make the connection between cores and vector~colorings:

\begin{theorem}[\cite{N71}]\label{lem:Nes}
Let $G$ be a connected graph. Every locally injective endomorphism of $G$ is an automorphism.
\end{theorem}

This allows us to prove the following simple lemma which is essential to our results on cores.

\begin{lemma}\label{lem:main}
If $G$ is a connected graph, then $G$ is a core if and only if there exists a (possibly infinite) graph $H$ such that $G \to H$ and every homomorphism from $G$ to $H$ is locally injective.
\end{lemma}
\proof
If $G$ is a core, then set $H = G$ and we are done. Conversely, suppose $G$ is connected and not a core. Further suppose that $G \to H$. We will show that there exists a homomorphism from $G$ to $H$ that is not locally injective. Since $G$ is not a core, there exists an endomorphism $\rho$ of $G$ which is not an automorphism. By Lemma~\ref{lem:Nes} $\rho$ is not locally injective. Let $\varphi$ be any homomorphism from $G$ to $H$. It is easy to see that $\varphi \circ \rho$ is a homomorphism from $G$ to $H$ that is not locally injective.\qeds

We can apply the above in the case of $H = \mathcal{S}^d_t$ to obtain our main result relating vector colorings to cores, presented as Theorem~\ref{thm:vect2core} below. Note that a vector coloring is injective (resp.~locally injective)  if it is  injective (resp.~locally injective)  as a homomorphism to $\mathcal{S}^d_t$ for some $d \in \mathbb{N}$ and $t \ge 2$. Equivalently,  a vector coloring is (locally) injective if it does not map any two vertices (at distance two from each other) to the same vector. 

\begin{theorem}\label{thm:vect2core}
Let $G$ be a connected graph. If every optimal (strict) vector coloring of $G$ is locally injective, then $G$ is a core. In particular, if $G$ is  UVC and its  unique vector coloring is locally injective, then $G$ must be a core. 
\end{theorem}

%\proof
%Set $\chiv(G)=t$ and let $\varphi_1: G \to \mathcal{S}^d_t$ be an optimal vector coloring of $G$. Given an endomorphism $\varphi_2: G\to G$, it follows by Lemma~\ref{lem:basic} that $\varphi_1\circ \varphi_2: G\to \mathcal{S}^d_t$ is {an} optimal vector coloring of $G$. By assumption it follows that $\varphi_1\circ \varphi_2$ is locally injective. In turn, this implies that $\varphi_2$ is locally injective and by Theorem \ref{lem:Nes} we get that $G$ is a core. 
%\qeds

We note that  in practice,  local injectivity does not seem to be a strong restriction. In fact all of the vector colorings discussed in this paper are  injective. Furthermore, in Section \ref{subsec:2walkreg}  we  study  a class of graphs which always have locally injective vector colorings. %Therefore, if such a graph is uniquely vector colorable, it must be a core.

%%%%%%%%%%%%%%%%%%%%%%%%%%%%%%%%%%%%%%%%%%%%

\subsection{Kneser graphs are cores}\label{sec:knesercores}
Using Theorem \ref{thm:vect2core} combined with our results on unique vector colorability we now proceed to show that several graph families  are cores.

\begin{corollary}
For $n \ge 2r+1$, the  graphs $K_{n:r}$ and  $qK_{n:r}$ are cores.
\end{corollary}

\begin{proof}
It was shown in \cite{UVC1} that for $n \ge 2r+1,$ both  $K_{n:r}$ and  $qK_{n:r}$ are UVC. Moreover, their canonical vector colorings are injective (e.g., see \eqref{eq:canonicalvc_kneser}  and \eqref{eq:leasteigkneser} in the Appendix).   The proof is concluded  using~Theorem~\ref{thm:vect2core}.% completes the~proof.
\end{proof}

As already mentioned in the introduction this result is well-known, e.g. see~\cite{HahnTardif}. Nevertheless, our proof is of independent interest as it does not rely Erd{\H o}s-Ko-Rado Theorem or   the No Homomorphism Lemma. 

The above corollary leaves open the case of the $q$-Kneser graphs $qK_{2r:r}$. In the case where $r = 2$, these graphs are transitive on non-edges, and one can use this to show that they are cores. On the other hand, we have shown computationally that $2K_{4:2}$ is \emph{not} UVC. {We conjecture} %In general, we suspect 
that the graphs $qK_{2r:r}$ are cores but are not UVC, however we have not been able to prove either claim.

%\anote{what happens to  qkneser  for 2r?} \lnote{I don't know. The graph $2K_{4:2}$ is not UVC but it is a core. I suspect this is what happens in general but have not been able to prove either.}

%We will see Kneser and $q$-Kneser graphs, as well as their vector colorings, again in Section~\ref{sec:equalchivec}. There we will use our results on vector coloring to prove a previously known result about homomorphisms between different Kneser graphs, as well as an analog of this result for $q$-Kneser graphs and a class of graphs from the Hamming scheme we will meet in the next section.

%%%%%%%%%%%%%%%%%%%%%%%%%%%%%%%%%%%%%%%%%%%%
%%%%%%%%%%%%%%%%%%%%%%%%%%%%%%%%%%%%%%%%%%%%
%%%%%%%%%%%%%%%%%%%%%%%%%%%%%%%%%%%%%%%%%%%%
%%%%%%%%%%%%%%%%%%%%%%%%%%%%%%%%%%%%%%%%%%%%

\subsection{Hamming graphs are cores}\label{sec:hamming}
%The graphs we will consider in this section are Cayley graphs for $\mathbb{Z}_2^n$. 
Consider an abelian group $\Gamma$ and inverse closed {connection set} $C \subseteq \Gamma \setminus \{0 \}$. The {\em Cayley graph}  corresponding to $\Gamma$ and $C$, denoted by ${\rm Cay}(\Gamma, C)$, has  as its vertex set the elements  of $\Gamma$ and two vertices  $a,b \in \Gamma$ are adjacent if $a - b \in C$. 

 In this section we focus on   Cayley graphs over  $\mathbb{Z}_2^n$ with group operation  bitwise XOR. We refer to the number of 1's in an element of $\mathbb{Z}_2^n$ as its \emph{weight}.
 %whose elements  are  the binary strings of length $n$. 
 %We  refer to the number of 1's in such a string as its  \emph{weight}.  We consider $\mathbb{Z}_2^n$  as a  group .  
 As a  connection set we take all elements of weight $k$,  for some  fixed $k \in [n]$, which we denote by $C_{n,k}.$ Note that the  graphs ${\rm Cay}(\Z_2^n, C_{n,k})$   lie in the binary Hamming scheme, specifically they are the distance $k$-graphs of the $n$-cube. Furthermore, note  that ${\rm Cay}(\Z_2^n, C_{n,k})$  is bipartite if $k$ is odd.
 % and so we  implicitly assume that $k$ is even throughout this section, unless otherwise noted. 
  Also, if $k \ne n$ and $k$ is even, this is a non-bipartite graph with  two isomorphic components corresponding to the even and odd weight elements. % of $\mathbb{Z}_2^n$.  
 We  denote the component consisting of the even-weight vertices by  $\hnk$. %Note that these vertices form a subgroup isomorphic to $\mathbb{Z}_2^{n-1}$ and this is therefore still a Cayley graph. 
 
 Our main result in this section is   that $\hnk$ is UVC for any even integer $k\in [n/2+1,n-1]$.
% if $k \ge \frac{n}{2} + 1$ and $k$ is even. 
 Note that $\hnk$ is arc transitive, i.e., any ordered pair of adjacent vertices can be mapped to any other such pair by an automorphism of $\hnk$. Therefore, $\hnk$ is 1-walk-regular %so we are able to make use of Theorem~\ref{thm:1walkreg}.
and thus  we can  use Theorem~\ref{thm:1walkreg} to show it  is UVC. For this, we need   to 
 determine the canonical vector coloring of $\hnk$  and show that  condition \eqref{eq:conicatinfinity} is satisfied.

%%%%%%%%%%%%%%%%%%%%%%%%%%%%%%%%%%%%%%%%%%%%
%\label{subsec:leasteigs}
%%%%%%%%%%%%%%%%%%%%%%%%%%%%%%%%%%%%%%%%%%%%

%We can use the above construction to find the eigenvectors of our graph $\hnk$, but first we need some notation. For $x,y \in \mathbb{Z}_2^n$, we denote by  $x \cdot y$  the inner product of $x$ and $y$ considered as vectors over $\mathbb{Z}_2$. In other words, $x\cdot y$ is 0 if $x$ and $y$ share an even number of entries equal to 1, and is 1 otherwise. We use the notation  $x^\bot := \{y \in \Z^n_2 : x \cdot y = 0\}$.

%If we let $C$ be the connection set of a Cayley graph over $\Z^n_2$, then for each element $a \in \Z^n_2$ we have that the map $x \mapsto (-1)^{a\cdot x}$ is a character of $\Z_2^n$. 

As a  first  step we calculate the least  eigenvalue of $\hnk$. As ${\rm Cay}(\Z_2^n, C_{n,k})$ consists of two-isomorphic connected components, the spectrum  of $\hnk$ coincides with the spectrum  of ${\rm Cay}(\Z_2^n, C_{n,k})$  which can be calculated as follows:  The eigenvectors of a Cayley graph {for an abelian group} can be constructed using the characters of the underlying group. In particular, if $\chi$ is a character   of  $\Gamma$, then the vector  $(\chi(a))_{a \in \Gamma}$ is an eigenvector for ${\rm Cay}(\Gamma, C)$ with eigenvalue $\sum_{c \in C} \chi(c)$. Moreover, ranging over all $|\Gamma|$  characters we get a full orthogonal set of eigenvectors. For  details on the spectra of Cayley graphs see~\cite{lovaszspectra}~or~\cite{cayleyspectra}.

Recall  that the characters of $\Z_2^n$ are given by    the functions  $\chi_a(x)= (-1)^{a\cdot x}$, for all $a \in \Z^n_2$. Throughout, for   $x,y \in~\mathbb{Z}_2^n$, we denote by  $x \cdot y$  the inner product of $x$ and $y$ considered as vectors over $\mathbb{Z}_2$. We also define $x^\bot$ to be the set $\{y \in \Z^n_2 : x \cdot y = 0\}$. %Moreover, as $V(\hnk)$ is defined as  the even weights elements of $\Z_2^n$, it follows that $\chi_a(x)=\chi_{\one+a}(x), \ \forall x\in~V(\hnk).$ Consequently, we can restrict our attention to elements of weight at most $n/2$.
Each character $\chi_a$ corresponds to an  eigenvector $v_a\in \Z_2^n$ of ${\rm Cay}(\Z_2^n, C_{n,k})$ given~by
%By the general theory for
 %From this we see that the vector $v_a$ 
\begin{equation}\label{eqn:eigvecs}
v_a(x) = (-1)^{a\cdot x}, \text{ for } x\in \Z_2^n,
\end{equation}
with corresponding eigenvalue 
\begin{equation}\label{eqn:eigvalues}
\sum_{c \in C_{n,k}} (-1)^{a\cdot c} = |C_{n,k} \cap a^\bot| - |C_{n,k} \setminus a^\bot| = \binom{n}{k}-2|C_{n,k} \setminus a^\bot|.%2|C \cap a^\bot| - |C|.
\end{equation}
Lastly, note that 
\be\label{eq:useful}
\begin{aligned}
v_a(x)& =v_{\one+a}(x),\quad  \forall x\in V(\hnk), \text{ and } \\
 v_a(x) &=-v_{\one+a}(x),\quad  \forall x\in {\Z_2^n \setminus V(\hnk)}.
\end{aligned}
\ee

By \eqref{eqn:eigvalues} we see  that the smallest eigenvalue of ${\rm Cay}(\Z_2^n, C_{n,k})$  corresponds to  the elements  $a\in\Z^2_n$  %which $|C \cap a^\bot|$ is minimized. Alternatively, they correspond to the elements $a$ 
%the number of elements of the connection set they share an odd number of ones with, i.e~those 
that maximize $|C_{n,k} \setminus a^\bot|$.   Finding the maximum value of $|C_{n,k} \setminus a^\bot|$ was already considered  by Engstr\"{o}m et al.~\cite{EFJT}. They gave and inductive proof of the bound in the theorem below, but we also need to determine when equality is attained in this bound. However, their proof can be easily modified to achieve this: simply include the claim about attainment in their induction hypothesis. Thus we have the following:

%If we let $C_{n,k}$ denote the set of all weight $k$ elements of $\mathbb{Z}_2^n$, then they show the following:
\begin{theorem}[\cite{EFJT}] \label{thm:EFJT}
For any even integer $k\in \left[{n+1\over 2},n\right]$
%For even $k$ such that $k +1 \le n \le 2k-2$
 we have that 
\be\label{eq:hnkupperbound}
\left|C_{n,k} \setminus a^\bot\right| \le \binom{n-1}{k-1}, \quad  \forall a \in \mathbb{Z}_2^n.
\ee
Moreover, equality is attained in \eqref{eq:hnkupperbound}  if $a$ has weight $1$ or~$n-1$. If $k \in \left[{n\over 2} +1,n-1\right]$, then these are the only elements where equality is attained.
\end{theorem}
%\begin{proof}
% \anote{why not  
% add the proof for completeness?} \lnote{Yeah, we probably should. I will have to write it later though.}
%\end{proof} 

%\begin{theorem}[\cite{EFJT}]\label{thm:EFJT}
%For even $k$ such that $k \le n \le 2k-1$,
%\[|C_{n,k} \setminus a^\bot| \le \binom{n-1}{k-1},\]
%for all $a \in \mathbb{Z}_2^n$.
%\end{theorem}
%
%Note that if $a$ is an element of weight one, then the elements of $C_{n,k} \setminus a^\bot$ are exactly those which have a one in the same place as $a$. There are exactly $\binom{n-1}{k-1}$ such elements and so the bound in the above inequality can be met. On the other hand, if $a$ is an element of weight two and $n = 2k-1$, then $C_{n,k} \setminus a^\bot$ has size $2\binom{2k-3}{k-1} = \binom{2k-2}{k-1}$. Thus, when $n = 2k-1$ the above bound is not met by elements of a unique weight. Also, if $k = n$, then any odd weight will achieve the above bound. These turn out to be the only exceptions however, and by only slightly modifying the proof of Engstr\"{o}m, F\"{a}rnqvist, Jonsson, and Thapper for the above theorem, one can obtain the following:
%\lnote{Should we actually give the modified proof?}

Based on  Theorem~\ref{thm:EFJT} we   now compute the canonical vector coloring of~$\hnk$.%and from that its vector chromatic number.

\begin{lemma}\label{lem:vector_coloring_hnk}
For  any  even integer $k\in \left[{n+1\over 2},n\right]$ we have that 
%For  any  even integer $k\in [n/2+1,n-1]$ we have that
\be
\lambda_{min}(\hnk)=\frac{n-2k}{k}\binom{n-1}{k-1},\quad \text{ and }\quad  \chiv(\hnk) = \frac{2k}{2k-n}.
\ee
Furthermore, for   any  even integer $k\in \left[n/2+1,n-1\right]$, the canonical vector coloring of $\hnk$ is given by $x\mapsto p_x\in \R^n$ where 
\be\label{eq:hnkcvc}
p_x(i) = {(-1)^{x_i}\over \sqrt{n}}, \quad  \forall i\in[n]. 
\ee
\end{lemma}

\begin{proof}
As previously noted, the least   eigenvalue of $\hnk$ is equal to  the least   eigenvalue of $ {\rm Cay}(\Z_2^n, C_{n,k})$. The latter    
  is equal to  $\frac{n-2k}{k}\binom{n-1}{k-1}$  by   \eqref{eqn:eigvalues}  and   Theorem~\ref{thm:EFJT}. 
 Furthermore, as $\hnk$ is 1-walk-regular,  Theorem~\ref{thm:1walkreg} $(i)$  implies~$\chiv(\hnk)=\frac{2k}{2k-n}.$

Next, consider an   even integer $k\in \left[n/2+1,n-1\right]$. By Theorem \ref{thm:EFJT},  the least eigenvalue of $ {\rm Cay}(\Z_2^n, C_{n,k})$ has  multiplicity $2n$. In particular,   a set  of orthogonal   eigenvectors  is  given by $\{v_{e_i}\}_{i=1}^n\cup~\{v_{\one+e_i}\}_{i=1}^n$.  For all $i\in [n]$ write  $v_{e_i}$  as $(x_i, y_i)^\trans$ where $x_i$ is the restriction of $v_{e_i}$ on $V(\hnk)$ and $y_i$ its restriction on $\Z_2^n \setminus V(\hnk)$. Using \eqref{eq:useful} it follows that  $v_{\one+e_i}=(x_i,-y_i)$ for all $i\in [n]$. As $\langle v_{e_i},v_{e_j}\rangle =\langle v_{e_i},  v_{\one+e_j}\rangle=0,$ for all  $i\ne j$, 
 the vectors $\{x_i\}_{i=1}^n$ are pairwise orthogonal. Furthermore, note  that the multiplicity of $\lambda_{min}(\hnk)$ as an eigenvalue of $\hnk$ is $n$ (because its multiplicity as an eigenvalue of $ {\rm Cay}(\Z_2^n, C_{n,k})$ is~$2n$). Thus, the vectors $\left\{ {x_i\over \sqrt{2^{n-1}}}: i\in [n] \right\}$ form  an orthonormal basis of the least eigenspace of  $\hnk$. Lastly, according to Definition \ref{def:canonicalvc}, to construct the canonical vector coloring of $\hnk$ we consider the vectors $\left\{ {x_i\over \sqrt{2^{n-1}}}: i\in [n] \right\}$ as columns of a matrix and then we scale  its rows  by $\sqrt{{2^{n-1}\over n}}$. This shows that the canonical vector coloring  of $\hnk$ is  given by~\eqref{eq:hnkcvc}.
 %$a\in \{e_i: i\in [n]\}\cup \{\one+e_i: i\in [n]\}$. Nevertheless, we 
%The computation of the least eigenvalue of $\hnk$ follows from Equation~(\ref{eqn:eigvalues}), the bound in Theorem~\ref{thm:EFJT}, and the fact that this bound is always obtained by the weight 1 elements of $\mathbb{Z}_2^n$. 
\end{proof}

Lastly, to show that $\hnk$ is UVC, we must show that its canonical  vector coloring %given in  \eqref{eq:hnkcvc} 
satisfies \eqref{eq:conicatinfinity}.
%the hypotheses of~\ref{cor:1walkreg}. 
This is   accomplished in the following lemma.

\begin{lemma}\label{lem:condvi}Let $x\mapsto p_x\in \R^n$ be the canonical vector coloring of $\hnk$. % given  in \eqref{eq:hnkcvc}. 
Then, for any  $n \times n$ symmetric matrix $R$ we have that 
 \be\label{eq:fwefefe}
 p_x^\trans R p_y = 0, \text{ for all }  x \simeq y \Longrightarrow R = 0.
 \ee
 
\end{lemma}
\begin{proof}
Since  ${\rm span}\{p_x: x\in V(\hnk)\}=\R^n$ we just need to  show that $Rp_x=~0$, for all  $x\in V(\hnk).$ For this consider   the subspace
\be 
V_x:= \spn\{p_y : y \simeq x\},
\ee
and note  that the hypothesis of~\eqref{eq:fwefefe} can be equivalently expressed as $Rp_x\in V_x^\perp$, for all $x\in V(\hnk)$.
Thus, if we can show that $V_x=\R^n$, for all $x\in V(\hnk)$, we get from  Equation~\eqref{eq:fwefefe} that  $Rp_x = 0$ for all $x\in V(\hnk)$, and we are  done.

%It remains to  show that $V_x=\R^n$, for all $x\in \Z_2^n$.

 We first consider the case of $V_x$ when $x = \zero$, the vector of all zeros in $\Z^n_2$. %First we  show that $R$ must map $p_0$ to the zero vector. Then we will show that the same holds for all $p_x$ by symmetry. We consider the following subspace:
%We will show that $V_0 = \R^n$. 
The neighbors of $\zero$ are all the vectors of weight $k$ in $\Z_2^n$. For each pair of distinct $i,j \in [n]$, there exist weight $k$ vectors $y,z \in \Z_2^n$ such that $e_i-e_j=\frac{\sqrt{n}}{2}(p_y - p_z)$. The vectors $y$ and $z$ can be chosen by picking any two weight $k$ vectors that differ only in positions $i$ and $j$. Therefore, $e_i-e_j\in V_\zero$ for all $i,j \in [n]$.
%\[V_0 \ni \frac{1}{2}(p_x - p_y) = e_i - e_j.\]
Since ${\rm span}\{e_i - e_j : i \ne j\}={\rm span}(\one)^\perp$, we have that ${\rm span}(\one)^\perp\subseteq V_\zero$. Lastly, 
 %span the subspace of vectors orthogonal to $\one$ (the all ones vector). 
 as $\one = \sqrt{n}  p_\zero \in V_\zero$ (recall \eqref{eq:hnkcvc}) it follows that  $V_\zero = \R^n$. %Since $Rp_0$ must be orthogonal to $V_0$, we have that $Rp_0 = \zero$.

Next, consider an arbitrary $x \in \Z^n_2$. Note that $V_x={\rm Diag}(p_x)V_\zero$, where 
%We have that
%\[V_x = \spn\{p_{x+y} : y = 0 \text{ or } y \sim 0\}.\]
$\text{Diag}(p_x)$
 is  the diagonal matrix with  entries corresponding to  $p_x$. %on its diagonal, then
%\[p_{y+x} = D_x p_y\]
%for all $y$. 
%Therefore, $V_x$ is equal to the image of $V_0$ under the map $D_x$. 
As  $V_\zero = \R^n$ and the matrix $\text{Diag}(p_x)$ is  invertible, we have that $V_x = \R^n$, for all $x \in~\Z^n_2$. 
%Therefore $Rp_x = \zero$ for all $x$, and thus $R = 0$.
\end{proof}

%Combining Lemma \ref{lem:condvi}  with  Theorem~\ref{thm:1walkreg} we have:
Putting everything together we get:
\begin{theorem}\label{thm:hnk_is_uvc}
%For even $k$ satisfying $k+1 \le n \le 2k-2$, 
The graph $\hnk$ is~UVC for any even integer $k\in [n/2+1,n-1]$.
\end{theorem}

It is worth noting that Theorem \ref{thm:hnk_is_uvc} does  not hold for all even values of $k$. {It is not difficult to show that} for $n = 2k-1$, the weight two elements of $\mathbb{Z}_2^n$ also give eigenvectors corresponding to the least eigenvalue of $\hnk$. Moreover, {one can} show that these eigenvectors can be used to construct a different optimal vector coloring of $\hnk$.  Therefore $\hnk$ is not uniquely vector colorable for $n = 2k-1$ for any even $k$.

 The canonical vector coloring of $\hnk$ given in \eqref{eq:hnkcvc} is  injective (so in particular it is locally injective). Combining Theorem \ref{thm:hnk_is_uvc} and Theorem \ref{thm:vect2core} we~get:

\begin{corollary}
%For even $k$ with $k+1 \le n \le 2k-2$, 
The graph $\hnk$ is a~core for any even integer $k\in [n/2+1,n-1]$.
\end{corollary}

For even $k < n/2 + 1$, the situation is unclear, but a few special cases are settled. For instance, for $k = 2$ it is known that $\hnk$ is a core if and only if $n$ is not a power of two. Also, if $k = n/2$ then the vertices $x$ and $\one + x$ have the same neighborhood and thus $\hnk$ is not a core in this case. We noted above that for $n = 2k-1$ the graph $\hnk$ is never UVC. In this case $\hnk$ may or may not be a core. In particular, by the above argument we see that $\hnk$ is a core when $k = 2$ and $n = 5$, however by direct computations we have found that the core of $H_{7,4}$ is the complete graph on 8 vertices.

\subsection{2-walk-regular graphs}\label{subsec:2walkreg}
%%%%%%%%%%%%%%%%%%%%%%%%%%%%%%%%%%%%%%%%%%%%

A graph $G$ is said to be \emph{2-walk-regular} if it is 1-walk-regular with the additional property that, for all $k \in \mathbb{N}$, the number of walks of length $k$ with initial and final vertices at distance two from each other does not depend on the specific pair of vertices. In this section we show  that, with a few simple exceptions, any uniquely vector colorable 2-walk-regular graph must be a core.

To show this we need to define the \emph{distance 2-graph} of a graph $G$. This is the graph with vertex set $V(G)$ in which two vertices are adjacent if they are at distance 2 in $G$. We  denote this graph as $G_2$. Using this notion we can give another definition of 2-walk-regular graphs: a graph $G$ is 2-walk-regular if it is 1-walk-regular and there exist numbers $c_k$ for all $k \in \mathbb{N}$ such that $A^k \circ A_2 = c_k A_2$, where $A_2$ is the adjacency matrix of $G_2$.

The following lemma gives  a relationship between a graph and its distance 2-graph which we  need for the main result of this section.

\begin{lemma}\label{lem:dist2comps}
Let $G$ be a connected graph. The components of $G_2$ induce independent sets in $G$ if and only if $G$ is bipartite or complete multipartite.
\end{lemma}
\proof
It is easy to see that if $G$ is a connected bipartite or complete multipartite graph, then the components of $G_2$ induce independent sets in $G$.

To see the converse suppose that $G$ is connected, not bipartite and the components of $G_2$ induce independent sets in $G$. We  show that $G$ must be complete multipartite. Let $D_1, \ldots, D_k$ be the vertex sets of the components of $G_2$. Since these are independent sets in $G$, coloring vertices in $D_j$ with color $j$ gives a proper coloring of $G$. Since $G$ is not bipartite, we have that $k \ge 3$.

 We show that any shortest path  in $G$ only contains two colors and these alternate along the path.
Indeed, consider a shortest path in $G$ which contains  three or more colors. Note that consecutive vertices receive different colors (as color classes are independent sets) and thus, there  exist three consecutive vertices on this path with distinct colors. However this is a contradiction, since the first and last of these three vertices would be at distance two, and must therefore receive  the same color (as they lie in the same component of $G_2$).

This fact has two useful consequences. First, 
%The above in particular implies that
 if a vertex has two neighbors of distinct colors, then they must be adjacent.
 %if not, since they are at distance two they would get the same color
   Second,  every vertex has a neighbor of every color other than its own. To see this let $v\in V(G)$ and consider another vertex $u$ with a different color. Then the  neighbor of $v$ on the shortest path from  $v$  to $u$ (this exists as $G$ is connected) 
   has the required property. 
   
  % Indeed, 
   %As $G$ is connect, any vertex $v$ has a shortest path to any vertex of a different color $j$
     %$G$ is connected and therefore any vertex $v$ has a shortest path to any vertex of a different color $j$, and the neighbor of $v$ on this path must be colored $j$.

Lastly,  towards a contradiction suppose that $G$ is not complete multipartite. This implies there must be two vertices $u_1$ and $u_2$ in different color classes  (say colored 1 and 2 respectively)  that are not adjacent. Since they are not adjacent, $u_1$ and $u_2$ must be at distance at least two. However, by the above, any shortest path between them alternates colors and therefore they must be at distance at least three. Furthermore, by considering the fourth vertex on this path, we can assume that $u_1$ and $u_2$ are at distance exactly three. Therefore, there exist vertices $v_1$ and $v_2$ such that $u_1 \sim v_2 \sim v_1 \sim u_2$ (note that the subscripts of these vertices correspond to their colors). By the above, $v_2$ has a neighbor $w$ of color $3.$ Since $w$ and $u_1$ are vertices of different colors in the neighborhood of $v_2$, by the previous paragraph $w$ and $u_1$ must be adjacent. Similarly, $w$ and $v_1$ are adjacent. This leads to a contradiction. 
%Further, since $v_2 \sim u_1, v_1$, we have that $w \sim u_1, v_1$. However this implies that $u_1, w , v_1, u_2$ is a shortest path from $u_1$ to $u_2$ containing three distinct colors, a contradiction.
% Let $P=(u_1, v_2, v_1,\ldots, u_2)$ be the  shortest path from $u_1$ to $u_2$ and note it has length at least 3.  By the discussion in the previous paragraph $v$ has a neighbor $w$ of color 3  and since $v_2 \sim u_1, v_1$, we have that $w \sim u_1, v_1$.  Then, the path  $P'=(u_1, w, v_1,\ldots, u_2)$ is  another shortest path from $u_1$ to $u_2$ (as it has the same length with $P$) that contains 3 colors, a contradiction by the discussion in the third paragraph.
%Without loss of generality we may assume that vertices $u_1$ and $u_2$ 
%are at distance three from each other, and that they are colored 1 and 2 respectively. Since $u_1$ and $u_2$ are at distance three from each other, there exist vertices $v_1$ and $v_2$ such that $u_1 \sim v_2 \sim v_1 \sim u_2$ (note that the subscripts of these vertices correspond to their colors). By the above, $v_2$ has a neighbor $w$ of color 3. Further, since $v_2 \sim u_1, v_1$, we have that $w \sim u_1, v_1$. However this implies that $u_1, w, v_1, u_2$ is a shortest path from $u_1$ to $u_2$ containing three distinct colors, a contradiction.
\qeds

Using the above lemma, we are able to show that the canonical vector coloring of a 2-walk-regular graph is always locally injective.

\begin{lemma}\label{cdsfvweve}
Let $G$ be a connected 2-walk-regular graph that is not bipartite or complete multipartite. The canonical vector coloring of $G$ is locally injective.
\end{lemma}
\begin{proof}
Let $i \mapsto p_i$ be the canonical vector coloring of $G$. Recall that the  Gram matrix of this vector coloring is a scalar multiple of the projection, $E_\tau$, onto the eigenspace of $G$ corresponding to its least eigenvalue. Since $E_\tau$ is a polynomial in the adjacency matrix of $G$ and $G$ is 2-walk-regular, there exists a real number $c$ such that $E_\tau \circ A_2 = cA_2$, where $A_2$ is the adjacency matrix of $G_2$. Therefore, $\langle p_i, p_j\rangle $ is constant for all vertices $i$ and $j$ at distance 2 in $G$.

Suppose that  $i \mapsto p_i$ is not locally injective. Then there exist $i,j \in~V(G)$ that are at distance two in $G$ such that $p_i = p_j$. This means that $\langle p_i, p_j\rangle  = 1$, and by the  argument in the first paragraph  this implies that any pair of vertices at distance two are mapped to the same vector. Therefore, the vertices in a single component of $G_2$ are all mapped to the same vector. However,  by Lemma~\ref{lem:dist2comps} and the assumption,  $G_2$ has a component which contains a pair of adjacent vertices,  and this pair of vertices cannot be mapped to the same vector since their inner product must be negative. This gives a contradiction and proves the~theorem.
\end{proof}

The following theorem is a direct consequence of Lemma \ref{cdsfvweve}.
%The above immediately implies the following corollary:

\begin{theorem}\label{thm:2walkreg}
Let $G$ be a 2-walk-regular, non-bipartite, and not complete multipartite graph. If $G$ is uniquely vector colorable, then $G$ is a core.
\end{theorem}

Note that we do not need to assume that $G$ is connected in Theorem~\ref{thm:2walkreg}  since this is implied by unique vector colorability. Examples of 2-walk-regular graphs include 2-arc-transitive graphs, distance regular graphs, and in particular strongly regular graphs, which we  focus on in Section~\ref{sec:SRGcomps}.

\subsection{Taylor graphs}\label{sec:Taylor}

%\lnote{It would be really good to have picture showing the distance partition of a Taylor graph. Unfortunately I am no good at making such pictures.}

A connected graph $G$ of diameter $d$ is \emph{distance regular} if there exist numbers $p_{ij}^k$ for $i,j,k = 0,1, \ldots, d$ such that for any pair of vertices $u,v$ at distance $k$ from each other, the number of vertices $w$ at distance $i$ from $u$ and distance $j$ from $v$ is equal to $p_{ij}^k$. This turns out to be equivalent to the existence of numbers $b_0, \ldots, b_{d-1}$ and $c_1, \ldots, c_d$ such that for any vertices $u,v$ at distance $i$ in $G$, the number of neighbors of $v$ at distance $i+1$ from $u$ is $b_i$ and the number of neighbors of $v$ at distance $i-1$ from $u$ is $c_i$. The array $\{b_0, \ldots, b_{d-1}; c_1, \ldots, d_d\}$ is known as the \emph{intersection array} of $G$ and it characterizes many of its properties, such as the eigenvalues of $G$ and the numbers $p_{ij}^k$ from above. Also note that the number $b_0$ is the valency of $G$.

Another useful property of a distance regular graph $G$ is that the span of the adjacency matrices of its distance graphs is equal to the algebra of polynomials of its adjacency matrix $A$. This implies that any polynomial in $A$ is constant on entries corresponding to pairs of vertices at some fixed distance (similar to 1- and 2-walk-regularity, but for any distance), and that the adjacency matrices of its distance graphs are polynomials in $A$. For a detailed account of distance regular graphs we refer the reader to~\cite{BCN}.

A \emph{Taylor graph} is a distance regular graph whose  intersection array  is given by $\{k,\mu, 1; 1, \mu ,k\}$, thus they have diameter three. 
%Taylor graphs are a rich class  of graph which includes the icosahedral graph and the Gosset graph among others. 
 Examples of (non-bipartite) Taylor graphs include 
the icosahedral graph and the Gosset graph. Moreover, given any strongly regular graph $G$ with parameters $(v,k,a,c)$ (see Section~\ref{sec:SRGcomps} for definition) where $k = 2c$, one can construct a non-bipartite Taylor graph as follows: Take two copies $G_1$ and $G_2$ of $G$, and add an edge between a vertex $u$ of $G_1$ and vertex $v$ of $G_2$ if the corresponding vertices of $G$ were distinct and non-adjacent. Finally, add a vertex adjacent to every vertex of $G_1$ and a vertex adjacent to every vertex of $G_2$. This will be a Taylor graph on $2v+2$ vertices with intersection array $\{v,v-k-1,1;1,v-k-1,v\}$.

The parameters of a Taylor graph  imply the number of vertices at distance 1,2, and 3 from a given vertex is $k$, $k$, and 1, respectively. Thus a Taylor graph has $2k+2$ vertices and every vertex has a unique vertex at distance three from~it. We  refer to such pairs as antipodes. Note that the antipode of a vertex $u$ is adjacent to every vertex at distance two from $u$. 
We will show that every Taylor graph is UVC and thus a core unless it is bipartite. First, we need to prove the following~lemma.%For any distance regular graph $G$, the adjacency matrix of the distance $i$-graph, $G_i$, of $G$ is a polynomial in the adjacency matrix of $G$. This implies that the eigenspaces of $G_i$ contain the eigenspaces of $G$. This allows us to prove the following lemma:

\begin{lemma}\label{lem:antipode}
Let $G$ be a Taylor graph. Then, in the canonical vector coloring of $G$, pairs of vertices at distance three are assigned antipodal vectors.
\end{lemma}
\proof
Let $A$ be the adjacency matrix of $G$ and $A_3$ the adjacency matrix of the distance $3$-graph of $G$, denoted $G_3$. Note that $G_3$ is isomorphic to the disjoint union of some number of $K_2$ graphs, and therefore has only two eigenvalues: $1$ and $-1$. Let $E_\tau$ be the projection onto the $\tau$-eigenspace of $G$ where $\tau$ is its least eigenvalue. For $u \in V(G)$ let $p_u$ be the vector assigned to $u$ in the canonical vector coloring of $G$. Recall that $E_\tau$ is a scalar multiple of the Gram matrix of the $p_u$. Let $d$ be the dimension of the $\tau$-eigenspace. Then $\tr(E_\tau) = d$, since the trace of a projection is equal to its rank. Furthermore, since $G$ is distance regular, all polynomials in $A$ have constant diagonal, and so all of the diagonal entries of $E_\tau$ must be equal to $d/n$, where $n$ is the number of vertices of $G$. We will show that the entries of $E_{\tau}$ corresponding to pairs of vertices at distance three are equal to $-d/n$, which will imply that vectors assigned to such pairs in the canonical vector coloring are antipodal.

Since $A_3$ is a polynomial in $A$, we have that $A_3E_\tau = \lambda E_\tau$ where $\lambda$ is some eigenvalue of $G_3$, i.e., is $\pm 1$. We will show that $\lambda = -1$. To do this, it suffices to show that any $\tau$-eigenvector of $G$ is a $-1$-eigenvector of $G_3$. Suppose that $z$ is a $\tau$-eigenvector of $G$. Then $z$ is an eigenvector of $G_3$ with eigenvalue $\pm 1$, since these are its only eigenvalues. Suppose for contradiction that $z$ is a 1-eigenvector for $G_3$. Since $G_3$ is a disjoint union of $K_2$'s whose edges are between antipodes of $G$, this implies that $z$ is constant on pairs of antipodes. Furthermore, since $z$ is a $\tau$-eigenvector of $G$, it is orthogonal to the all ones vector since this is a $k$-eigenvector of $G$. Thus the entries of $z$ sum to zero. Now consider any vertex $u \in V(G)$ such that $z_u \ne 0$ and let $S = \{v \in V(G) :  v \simeq u\}$ be the closed neighborhood of $u$. Then there are no pairs of antipodes contained in $S$ and no pairs of antipodes contained in $V(G) \setminus S$, since this is the closed neighborhood of the antipode of $u$. Thus the antipode relation is a bijection between $S$ and $V(G) \setminus S$. Therefore,
\[0 = \sum_{v \in V(G)} z_v = \sum_{v \in S} z_v + \sum_{v \in V(G) \setminus S} z_v = 2\sum_{v \in S} z_v.\]
This implies that $z_u + \sum_{v \sim u} z_v = 0$ and thus $(Az)_u = \sum_{v \sim u} z_v = -z_u$. Therefore, $z$ is a $-1$-eigenvector of $G$. But $\tau \ne -1$ since it is well known that the only connected graphs with least eigenvalue equal to $-1$ are the complete graphs. Thus $z$ cannot be a $\tau$-eigenvector of $G$, a contradiction.

By the above, we have that $A_3E_\tau = -E_\tau$. Let $\text{sum}(M)$ denote the sum of the entries of the matrix $M$, and note that $\text{sum}(M \circ N) = \tr(MN)$ for any symmetric matrices $M$ and $N$. Thus we have that
\begin{equation*}
\text{sum}(A_3 \circ E_\tau) = \tr(A_3E_\tau) = \tr(-E_\tau) = -d.
\end{equation*}
Since $E_\tau$ is a polynomial in $A$ and $G$ is distance regular, the entries of $E_\tau$ corresponding to pairs of vertices at distance three are all equal to some constant $\gamma$. The number of such entries is equal to the number of 1's in $A_3$ which is twice the number of edges of $G_3$. Since $G_3$ is the disjoint union of $K_2$'s, this is just $n$, the number of vertices of $G$. Therefore, $n\gamma = \text{sum}(A_3 \circ E_\tau) = -d$ and thus $\gamma = -d/n$, which is the negative of the diagonal entries of $E_\tau$. Thus for vertices $u$ and $v$ at distance three, $\langle p_u,p_v\rangle = -\langle p_u,p_u\rangle = -1$, and this implies that $p_u = -p_v$.\qeds

Using the above lemma, we can show that every Taylor graph is UVC.

\begin{theorem}
Any Taylor graph is uniquely vector colorable. Furthermore, this implies that any non-bipartite Taylor graph is a core.
\end{theorem}
\proof
Let $G$ be a Taylor graph and let $u \mapsto p_u \in \mathbb{R}^d = \spn\{p_v : v \in V(G)\}$ be its canonical vector coloring. In order to prove that $G$ is UVC, we must show that the only symmetric matrix $R$ satisfying $p_u^\trans R p_v = 0$ for $u \simeq v$ is the zero matrix. Consider the subspace
\[V_u = \spn\{p_w : w \simeq u\}.\]
We will show that $V_u = \mathbb{R}^d$. Let $v$ be the antipode of $u$. %Obviously, $p_u \in V_u$ and so by Lemma~\ref{lem:antipode} $p_{v} = -p_u \in V_u$ where $v$ is the antipode of $u$.
Consider a vertex $x \in V(G)$ whose antipode is $y \in V(G)$. If $x \simeq u$, then $p_x \in V_u$ by definition, and we are done. Otherwise we must have $y \simeq u$, since $V(G) \setminus \{w \in V(G) : w \simeq u\} = \{w \in V(G) : w \simeq v\}$ and it is not possible for both $x$ and $y$ to be contained in the closed neighborhood of $v$ because they are at distance three. If $y \simeq u$ then $-p_x = p_y \in V_u$ and thus $p_x \in V_u$. Thus $p_x \in V_u$ for all $x \in V(G)$ and therefore $V_u = \mathbb{R}^d$, and this holds for all $u \in V(G)$. %Recall the argument we used to show that the graph $G'$ in the proof of Lemma~\ref{lem:antipode} was complete. We showed that for any vertex $u'$ not equal to $u$ or $v$, either $u \sim u'$ or $v \sim u'$. In the former case, we have that $p_{u'} \in V_u$. In the latter case, the antipode $v' \not\in \{u,v\}$ of $u'$ cannot be in the neighborhood of $v$, since then it would be at distance at most two from $u'$. Therefore $v'$ must be in the neighborhood of $u$ (i.e., $u \sim v'$), and so by Lemma~\ref{lem:antipode} we have that $p_{u'} = -p_{v'} \in V_u$. Thus we have shown that $p_w \in V_u$ for all $w \in V(G)$ and therefore $V_u = \mathbb{R}^d$, and this holds for all $u \in V(G)$.

The equation $p_v^\trans R p_u = 0$ for $u \simeq v$ implies that for fixed $u$ the vector $Rp_u$ lies in $V_u^\bot$. By the above, this means that $Rp_u = 0$ for all $u \in V(G)$, and thus $R = 0$ as desired. This implies that any Taylor graph $G$ is UVC, and thus by Theorem~\ref{thm:2walkreg},  $G$ is a core unless it is bipartite or complete multipartite. Since complete multipartite graphs have diameter two, they are never Taylor graphs. Thus we have shown that a Taylor graph is a core unless it is bipartite (which is possible).\qeds

We remark that  bipartite Taylor graphs  are known as   \emph{crown graphs}, i.e., complete bipartite graphs with a perfect matching removed. In terms of the parameters of a Taylor graph this occurs whenever $\mu = k-1$.

\subsection{Computations for strongly regular graphs}\label{sec:SRGcomps}
Motivated by Theorem \ref{thm:2walkreg},  we now give an algorithm  for showing that a 2-walk-regular graph is UVC, and thus, a core.  This relies  on the following~result.

 %We have already used this approach to prove that entire families of 2-walk-regular graphs are cores (i.e., Kneser graphs, and $q$-Kneser graphs in Section \ref{sec:cores}, and Hamming graphs for certain parameter sets in Section \ref{sec:hamming}). 
 
 %Theorem  \ref{thm:1walkreg} also makes it possible to check whether a specific 2-walk-regular graph of interest, or a sporadic 2-walk-regular graph is uniquely vector colorable.  

\begin{lemma}\label{cor:coretest}
Let $G$ be a 1-walk-regular graph and let  $i\mapsto p_i\in \R^d$   be its  canonical vector coloring. Also, let $d$ be the multiplicity of the least eigenvalue of $G$. Then, $G$ is UVC if and only if
\[\dim\left(\spn\{p_e: e\in E(G)\}\right)={d+1\choose 2},\]
where \be
p_e:=p_ip_j^\trans+p_jp_i^\trans, \quad \text{ for all }  e=\{i,j\}\in E(G),
\ee
If $G$ is additionally 2-walk-regular then it is a core unless it is bipartite or complete multipartite.
\end{lemma}

\begin{proof}

Let $G$ be a 2-walk-regular graph, and let $i\mapsto p_i\in \R^d$  be its  canonical vector coloring. 
%Recall Condition \eqref{eq:conicatinfinity} given in Corollary \ref{cor:1walkreg} is: for all symmetric $d\times d$ matrices $R$, if $p_i^TRp_j=0$ for all $i=j$ or $\{i,j\}\in E(G)$, then $R=0$. 
% is satisfied if and only if
%%
%\[
%\begin{array}{rcl}
%0=p_i^TRp_j+p_j^TRp_i &=& \tr(p_i^TRp_j+p_j^TRp_i)\\
%&=& \tr(Rp_j p_i^T+Rp_i p_j^T)\\
%&=&\tr(R(p_j p_i^T+p_ip_j^T)).
%\end{array}
%\]
%
Also, recall that the canonical vector coloring satisfies
%that the $p_i$ are the rows of a matrix whose columns are $\tau$-eigenvectors of $G$. Therefore,
%
\be\label{vdbrth}
\tau p_i=\sum_{j\sim i}p_j, \quad \forall i\in [n],
\ee
%
% So $p_i^TRp_i=0$ if and only if
%%
%\[
%0=p_i^TRp_i=p_i^TR(\tau p_i)=p_i^TR\left(\sum_{j: j\sim i}p_j\right)=\sum_{j: j\sim i}p_i^TRp_j.
%\]
%
where $\tau$ is the least eigenvalue of $G$ (which is not zero). Thus, if $p_i^\trans Rp_j = 0$ for all $i \sim j$, it follows by \eqref{vdbrth}  that $p_i^\trans Rp_i = 0$ for all $i\in [n]$. 

By Theorem \ref{thm:1walkreg}, the graph $G$ is UVC if and only if    condition  \eqref{eq:conicatinfinity} holds, which by the previous discussion  can be equivalently expressed as  
\be\label{eq:conicatinfinity_simplified}
p_i^\trans R p_j ={\tr(R(p_j p_i^\trans+p_ip_j^\trans))\over 2}= 0, \text{ for all  } i \sim j \ \Longrightarrow\  R=0,
\ee
 In turn,  Equation \eqref{eq:conicatinfinity_simplified}  expresses  that the matrices     $\{p_e: e\in E(G)\}$ span the space of  symmetric $d\times d$ matrices, which has dimension $\binom{d+1}{2}.$
 %(i.e., the orthogonal complement of the span of the $p_e$ is the zero matrix). 
 %So we can check whether or not $G$ is UVC  simply by computing the dimension of $\spn\{p_e\,:\,e\in E(G)\}$.
% We summarize the discussion in the corollary below. 
The proof is concluded by Theorem~\ref{thm:2walkreg}. 
\end{proof}

To  use Lemma \ref{cor:coretest} we need to determine  the canonical vector coloring of $G$ and then  compute the matrices  $p_e$. This requires us to compute an orthonormal basis of the least eigenspace of $G$. However, these eigenvectors may contain irrational entries. Since we are interested in the dimension of the span of the $p_e$, our computations must be exact, rather than numerical. Thus, this approach may produce some computational difficulties. Instead, we use  a method of determining $\dim\left(\spn\{p_e: e\in E(G)\}\right)$ that avoids  eigenvector computations. The  details of the implementation are given in  Appendix \ref{appendix:comput}.

%We applied this procedure  to $73816$ strongly regular graphs obtained from Ted Spencer's webpage~\cite{Spence} 
%showing that 62168 (approx.~$84\%$) of them are UVC and therefore cores. %we  give details of the results of computations we carried out on strongly regular graphs. 
%
  As a case study, we applied this method to investigate  how often a strongly regular graph happens to be a core.
    The parameter set of a {\em strongly regular graph} (SRG) is a 4-tuple $(v,k,a,c)$ where $v$ is the number of vertices, $k$ is the degree of each vertex, $a$ is the number of common neighbors for every  pair of adjacent vertices, and $c$ is the number of common neighbors for each pair of non-adjacent vertices. 
SRGs are examples of 2-walk-regular graphs that are of significant  interest to graph theorists. 

The characterization of the cores of SRGs is the subject of a conjecture of Cameron and Kazinidis~\cite{symmetriccores} that was recently verified by Roberson~\cite{RobersonSRG}: The  core of any strongly regular graph is either itself or a complete~graph.

%Another consideration is computing the rank of $M$. 
%%%%%%%%%%%%%%%%%%%%%%%%%%%%%%%%%%%%%%%%%%%%%%%%
%\subsection{Strongly Regular Graphs}\label{subsec:srgs}
%%%%%%%%%%%%%%%%%%%%%%%%%%%%%%%%%%%%%%%%%%%%%%%%
 
 For our data set, we used Ted Spence's list of SRGs available online~\cite{Spence}. The data from these computations is summarized in Table \ref{srgtable} in  Appendix \ref{appendix:comput}.  Overall, approximately 84\% of the strongly regular graphs we tested were UVC and therefore cores. A natural question is how many of the non-UVC graphs are cores. By the result of Roberson~\cite{RobersonSRG}, a SRG  is a core if and only if its clique number is \emph{not} equal to its chromatic number. Using this we verified that only 79 of the 73816 strongly regular graphs we considered are not cores. This shows that almost $99.9\%$ of all considered instances  were~cores. 
\section{Homomorphisms of graphs with $\chiv(G) =\chiv(H)$}\label{sec:equalchivec}
%%%%%%%%%%%%%%%%%%%%%%%%%%%%%%%%%%%%%%%%%%%%
%%%%%%%%%%%%%%%%%%%%%%%%%%%%%%%%%%%%%%%%%%%%

In this section we  give necessary and sufficient conditions for the existence   of homomorphisms between two graphs with  equal  vector chromatic numbers. Our main tool is the following  result. 
% whose proof is omitted.  %consequence of  Lemma \ref{lem:basic}. 

 \begin{lemma}\label{lem:submatrix}
Consider    two graphs  $G$ and $H$ where $G\to H$, $G$ is UVC and $\chiv(G)=\chiv(H)$.  If $\varphi_1$ is an optimal vector coloring of $H$  and $\varphi_2$  is the unique optimal vector coloring of $G$, we have that % from $G$ to $H$, then
$$\{\langle \varphi_2(g),  \varphi_2(g')\rangle: g, g' \in V(G)\}\subseteq \{\langle \varphi_1(h), \varphi_1(h')\rangle: h, h' \in V(H)\}.$$
%Furthermore,   let $M$  be the Gram matrix of an  optimal vector coloring of $H$. Given a homomorphism $\varphi: G\to H$,   
% the principal submatrix of $M$ corresponding to  $\{\varphi(g)\}_ {g\in V(G)}$ is the Gram matrix of the unique 
%then the $|V(G)|\times |V(G)| $ submatrix  $M'$ given by  $M'_{g,g'}:=M_{\phi(g),\phi(g')}\, \forall g,g'\in V(G)$ 
\end{lemma} 
\begin{proof}Let $\varphi$ be a homomorphism $G\to H$. Since $\chiv(G)=\chiv(H)$, 
 the map  $\varphi_1\circ\varphi$ is  an optimal  vector coloring of $G$. Lastly, as  $G$ is UVC we have that 
 $$\langle \varphi_2(g),  \varphi_2(g')\rangle: g, g' \in V(G)\}=\{\langle (\varphi_1\circ \varphi)(g),  (\varphi_1\circ \varphi)(g')\rangle: g, g' \in V(G)\},$$
 and the latter set is clearly contained in $\{\langle \varphi_1(h), \varphi_1(h')\rangle: h, h' \in V(H)\}.$
\end{proof}
 %As usual, we mostly consider the case where both $G$ and $H$ are UVC. 
As we now show, this simple observation yields some algebraic conditions between $G$ and $H$ which allows us {to restrict} the {possible} homomorphisms~$G\to~H$.

\subsection{Kneser graphs}\label{sec:kneser}
%%%%%%%%%%%%%%%%%%%%%%%%%%%%%%%%%%%%%%%%%%%%
As already mentioned in the introduction, 
%\lnote{I think it is an open question what are all the possible homomorphisms among Kneser graphs, so maybe we should say something about that, even though our work only applies to the cases where the two Kneser graphs have the same fractional/vector chromatic number.}
%The fractional chromatic number of the Kneser graph $K_{n:r}$ is equal to $n/r$. Therefore, if $K_{n:r} \to K_{n':r'}$, we must have that $n/r \le n'/r'$. 
 Stahl  used the Erd\H{o}s-Ko-Rado Theorem to show that if $n/r = n'/r'$, then $K_{n:r} \to K_{n':r'}$ if and only if $n'$ is an integer multiple of $n$ (in which case $r'$ is an integer multiple of $r$ as well)~\cite{stahl}. Since $\chiv(K_{n:r})=n/r$, we can apply Lemma~\ref{lem:submatrix}  to obtain an alternative proof of this result. %The vector and fractional chromatic numbers of $qK_{n:r}$ are both $[n]_q/[r]_q$, but there is no analogous result to the above for these graphs. However, we will prove a necessary condition for the existence of a homomorphism between two $q$-Kneser graphs with the same vector chromatic numbers. Moreover, we will prove an analog of the if and only if result for Kneser graphs for the Hamming graphs introduced in Section~\ref{sec:hamming}.

%%%%%%%%%%%%%%%%%%%%%%%%%%%%%%%%%%%%%%%%%%%%

\begin{theorem}[\cite{stahl}]\label{thm:kneser_hom}
Let $n,r,n',r'$ be integers satisfying $n > 2r$ and $n/r = n'/r'$. Then there exists a homomorphism from $K_{n:r}$ to $K_{n':r'}$ if and only if $n'$ and $r'$ are integer multiples of $n$ and $r$ respectively.
\end{theorem}
\begin{proof}
If $n' = mn$ we have that  $r' = mr$ (as $n/r = n'/r'$). To show that $K_{n:r} \to~
K_{n':r'}$  we  consider the vertex set of $K_{n':r'}$ to be the $r'$-subsets of $[m] \times [n]$.
 % It is clear this makes no difference. 
The desired homomorphism   maps any $r$-subset $S\subseteq [n]$ to $[m] \times S$. %It is easy to see that this is a homomorphism.

Conversely,  consider a homomorphism   $\varphi: K_{n:r} \to K_{n':r'}$.  By assumption $\gamma:=n/r = n'/r'$ and thus these  two graphs have the same vector chromatic numbers. %Consequently,  we can apply Lemma~\ref{lem:Gramentries}. 
Given two sets $S,S'\subseteq [n]$ with $|S\cap S'|=k$  it follows by  \eqref{eq:canonicalvckneser} that 
% Set  $\gamma:= n/r = n'/r'$.   %Recalling the unique optimal vector coloring of $K_{n:r}$ described in Section~\ref{subsec:kneser}, 
%Similarly, the inner product of the vectors assigned to two vertices of $K_{n':r'}$ whose intersection has size $k'$ is
\be
 \langle p_S, p_{S'}\rangle=\frac{k}{r}\cdot \frac{\gamma}{\gamma - 1} - \frac{1}{\gamma - 1},
 \ee
where $S\mapsto p_S$ is the canonical  vector coloring of $K_{n:r}$.
By Lemma~\ref{lem:submatrix} we~have 
\be\label{eq:kneserinclusion}
\left\{ \frac{k}{r}\cdot \frac{\gamma}{\gamma - 1} - \frac{1}{\gamma - 1}: k \in [r] \right\}\subseteq \left\{ \frac{k'}{r'}\cdot \frac{\gamma}{\gamma - 1} - \frac{1}{\gamma - 1}: k'\in [r']\right\}.
\ee
%\[\{\mu_k : k \in [r] \} \subseteq \{\mu'_{k'} : k' \in [r']\}.\]
In particular,  it follows by \eqref{eq:kneserinclusion} that for $k=1$ %we see that 
%the fact that $\mu_1$ is contained in the latter set implies that 
there exists a $k' \in [r']$ such that
\[\frac{1}{r}\cdot \frac{\gamma}{\gamma - 1} - \frac{1}{\gamma - 1} = \frac{k'}{r'}\cdot \frac{\gamma}{\gamma - 1} - \frac{1}{\gamma - 1}.\]
This holds if and only if $1/r = k'/r'$ which is equivalent to $r' = k'r$. Therefore $r'$ is an integer multiple of $r$, and thus $n'$ is an integer multiple of $n$.
\end{proof}

\subsection{$q$-Kneser graphs}\label{sec:qknesernesssufficient}
In this section we give  a necessary condition for the existence of homomorphisms between $q$-Kneser graphs. Since  $\chiv(qK_{n:r})=[n]_q/[r]_q$ we can again use Lemma~\ref{lem:submatrix}.  In fact our necessary condition is  completely analogous to Theorem~\ref{thm:kneser_hom}. The only change one needs to make is to replace $n,r$, and $k$ with their $q$-analogues $[n]_q, [r]_q$, and $[k]_q$ respectively, noting also that $[1]_q = 1$.

%\anote{was the following theorem  known?}
\begin{theorem}\label{thm:qknesernessecary}
Let $n,r,q,n',r',q'$ be integers satisfying $n > 2r$, $n' > 2r'$, and $[n]_q/[r]_q = [n']_{q'}/[r']_{q'}$. If $qK_{n:r} \to q'K_{n':r'}$, then
\[\left\{\frac{[k]_q}{[r]_q} : k \in [r]\right\} \subseteq \left\{\frac{[k']_{q'}}{[r']_{q'}} : k' \in [r']\right\}.\]
In particular, $[n']_{q'}$ and $[r']_{q'}$ are integer multiples of $[n]_q$ and $[r]_q$ respectively.
\end{theorem}
As the proof of this fact  is quite similar to Theorem \ref{thm:kneser_hom}  we omit it. %we will not repeat it.
Unfortunately, we do not know how to prove a necessary and sufficient condition  for $q$-Kneser graphs. It was shown in~\cite{chowdhury} that there is a homomorphism from $q^mK_{n:r}$ to $qK_{mn:mr}$, but it is not clear if these are the only homomorphisms between $q$-Kneser graphs with the same vector chromatic number.

%\lnote{Can we turn the above theorem into an if and only if? It seems like our necessary condition is quite strong. Another way of saying it is that
%\[\frac{[n']_{q'}}{[n]_{q}} = \frac{[r']_{q'}}{[r]_{q}}\]
%and for every $k \in [r]$ there exists a $k' \in [r']$ such that
%\[\frac{[n']_{q'}}{[n]_{q}} = \frac{[r']_{q'}}{[r]_{q}} = \frac{[k']_{q'}}{[k]_{q}}.\]
%This seems pretty restrictive but I don't think I know enough about these $q$-numbers to figure anything out.}

%\lnote{We can actually get an if and only if for when a $q$-Kneser graph $qK_{n:r}$ has a homomorphism to a Kneser graph $K_{n':r'}$ with the same vector chromatic number. In this case the same argument as above shows that $n'$ and $r'$ are integer multiples of $[n]_q$ and $[r]_q$, but this is sufficient as well because $qK_{n:r}$ is an induced subgraph of $K_{[n]_q:[r]_q}$. To see this use the lines of $\mathbb{F}_q^n$ as the ground set for the vertices of $K_{[n]_q:[r]_q}$.}

\subsection{Hamming graphs}\label{sec:hamming_homs}
In this section we focus on  the graphs $\hnk$ studied in Section~\ref{sec:hamming}.  %these graphs have
By Lemma~\ref{lem:vector_coloring_hnk} we have that
$\chiv(\hnk) = \frac{2}{2 - \frac{n}{k}}.$ 
Consequently, $\chiv(\hnk) = \chiv(H_{n',k'})$ if and only if $n/k = n'/k'$.
 %and the same is true  for the graph $\hnk'$. 
 Moreover, we have seen in Theorem~\ref{thm:hnk_is_uvc}
%For even $k$ satisfying $k+1 \le n \le 2k-2$, 
that the graph $\hnk$ is~UVC for any even integer $k\in [n/2+1,n-1]$. Furthrermore, recall that the canonical vector coloring is given by $
p_x(i) = {(-1)^{x_i}\over \sqrt{n}}, $ for all $ i\in[n]$, and note that 
\begin{equation}\label{eqn:distcol}
\langle p_x,  p_y\rangle = \frac{n - 2d(x,y)}{n} = 1 - 2\frac{d(x,y)}{n},
\end{equation}
where $d(x,y)$ is the Hamming distance of two vertices of $\hnk$.   Next we   use Lemma \ref{lem:submatrix} to characterize homomorphisms  $\hnk \to H_{n',k'}$, when~{$n/k = n'/k'$}.

%\begin{theorem}\label{thm:necessaryham}
%Consider integers  $n,k,n',k'$ where ${n+1\over 2}<k<n$, %  $k < n < {2k-1}$,
%  $n/k = n'/k'$ and  both $k$ and $k'$ are even. Furthermore, consider graphs $G,G'$ such that  
%\be\label{eq:twiinequalities}
%\hnk \subseteq G \subseteq \hnk' \ \text{ and }\ H_{n',k'} \subseteq G' \subseteq H'_{n',k'}
%\ee
% If $G \to G'$, then $n'$ and $k'$ are integer multiples of $n$ and $k$, respectively. Moreover, any such homomorphism is injective and therefore $G$ is a subgraph of $G'$.
%\end{theorem}

%Lastly, the above necessary condition is also sufficient for  $\hnk$ and $H'_{n,k}$.

\begin{theorem}\label{thm:homomor_hnk}
Consider integers  $n,k,n',k'$ where   $k < n < {2k-1}$,  $n/k = n'/k'$ and  both $k$ and $k'$ are even.
%Suppose $n,k,n',k'$ are integers such that $k < n < 2k-1$ and $n/k = n'/k'$.
 Then, we have that  $\hnk \to H_{n',k'}$ if and only if $n'$ and $k'$ are integer multiples of $n$ and $k$, respectively. %Moreover, any  homomorphism $\hnk \to H_{n',k'}$ is an isomorphism from $\hnk$ to an induced subgraph of $H_{n',k'}$. 
\end{theorem}

\begin{proof}

Since $n/k = n'/k'$ we have   that $\chiv(\hnk) = \chiv(H_{n',k'})$.  
By our assumptions on $n$ and $k$ it follows by  Theorem~\ref{thm:hnk_is_uvc}  that both $ \hnk$ and  $H_{n',k'}$ are UVC. Furthermore, 
since $n > k$ and $n/k = n'/k'$, we have that $n' > k'$.  Also, $n < 2k - 1$ implies that $n' < 2k'$ and therefore $n' \le 2k' - 1$. Therefore, by 
 Lemma~\ref{lem:submatrix}~and Equation \eqref{eqn:distcol} we get that: 
\[\left\{\frac{d}{n} : d \in [n]\right\} \subseteq \left\{\frac{d'}{n'}: d' \in [n']\right\}.\]
In particular, for $d=1$ this implies that there exists $d' \in [n']$ such that
\[\frac{1}{n} = \frac{d'}{n'},\]
and thus $n' = d'n$ (and $k'=d'k$). %The second claim follows from Lemma~\ref{lem:injcol}~$(i)$  and the fact that the unique vector coloring of $G$ is injective.

 For the other direction, assume   there exists an integer $m$ such  that $n' = mn$ and $k' =mk$.
% $n'$ and $k'$ are integer multiples of $n$ and $k$, respectively.
%To see that this also suffices, we give a homomorphism.  
A  homomorphism  $\hnk \to H_{n',k'}$
%, and from $\hnk'$ to $E'_{n',k'}$,
 is given by mapping any element of $\Z_2^n$ %viewed as a binary string, 
 to $m$ copies of that vector  concatenated together. 
%Lastly, to see that any  homomorphism  $\hnk \to H_{n',k'}$ is an isomorphism from the former to an induced subgraph of the latter, it suffices to apply Lemma~\ref{lem:injcol} $(ii)$ (see Remark \ref{rem:faitful_equality}), noting that the unique  vector coloring of $\hnk$  satisfies \eqref{faithfulequality}.
%is injective and faithful.\anote{faithful remove}  %The unique vector coloring of $\hnk$ is not faithful according to our definition above. %However, it is faithful as a \emph{strict} vector coloring, and so similar reasoning applies.
\end{proof}
%Furthermore the analogous statement holds for $\hnk'$ and $H'_{n',k'}$.

\paragraph{Acknowledgements:}  D.~E.~Roberson was supported by Cambridge Quantum Computing Ltd.~and the EPSRC, as well as Simone Severini and Fernando Brandao. R.~\v{S}\'{a}mal was partially supported by grant GA \v{C}R 16-19910S and by grant LL1201 ERC CZ of the Czech Ministry of Education, Youth and Sports. A.~Varvitsiotis was supported in part by the Singapore National Research Foundation under NRF RF Award No. NRF-NRFF2013-13.

%%%%%%%%%%%%%%%%%%%%%%%%%%%%%%%%%%%%%%%%%%%%%
%%%%%%%%%%%%%%%%%%%%%%%%%%%%%%%%%%%%%%%%%%%%%
%\section{Discussion}\label{sec:discuss}
%%%%%%%%%%%%%%%%%%%%%%%%%%%%%%%%%%%%%%%%%%%%%
%%%%%%%%%%%%%%%%%%%%%%%%%%%%%%%%%%%%%%%%%%%%%

\appendix
\section{Canonical vector coloring of Kneser graphs }\label{appendix:kneser}
%%%%%%%%%%%%%%%%%%%%%%%%%%%%%%%%%%%%%%%%%%%%

%In this section we will introduce Kneser and $q$-Kneser graphs, which we previously showed to be uniquely vector colorable in~\cite{UVC1}.

 Given integers $n$ and $r$ such that $n \ge r$, the Kneser graph $K_{n:r}$ has as vertices the  $r$-subsets of $[n]$, and  two are adjacent if they are disjoint. The $q$-Kneser graph $qK_{n:r}$ has the $r$-dimensional subspaces of the finite vector space $\mathbb{F}_q^n$ as its vertices, and two of these subspaces are adjacent if they are \emph{skew}, i.e.,~if their intersection is the 0-subspace. Note that for $n < 2r$, neither of these graphs have any edges. For $n = 2r$, the Kneser graph $K_{n:r}$ is  a perfect matching, but $qK_{n:r}$ can have complicated  structure. Here we only consider the case $n \ge 2r + 1$.

Both the Kneser and $q$-Kneser graphs are vertex and edge transitive, and therefore  1-walk-regular. 
It was shown  in~\cite{UVC1} using Theorem \ref{thm:1walkreg}  that for $n \ge 2r+1$, the graphs $K_{n:r}$ and $qK_{n:r}$ are both UVC.
For completeness, we now give the  optimal vector 
%Here we will describe 
   colorings  from \cite{UVC1}. 
   
   %which were originally given in~\cite{UVC1}. %\anote{note really, we give optimal gram matrices}

For  $K_{n:r}$, the {coordinates of the} vectors in the  vector coloring  are  indexed by $[n]$. To a subset $S\subseteq [n]$ with $|S|=r$, we assign the unit vector $p_S\in \R^n$ given   by:
\be\label{eq:canonicalvc_kneser}
p_S(i)  = \begin{cases}
{ r-n\over \sqrt{nr(n-r)}}, & \text{if } i \in  S, \\
{ r\over \sqrt{nr(n-r)}}, & \text{otherwise}.
\end{cases}
\ee
The inner product of the vectors assigned to two $r$-subsets of $[n]$ depends only on the size of their intersection. Indeed, given two $r$-subsets $S,S'\subseteq [n]$ with $|S\cap S'|=k$ we have that 
%Let $\mu_k$ be the value of this inner product for two vertices of $K_{n:r}$ whose intersection has size $k$. Then

\be\label{eq:canonicalvckneser}
%\begin{align*}
\langle p_S, p_{S'}\rangle =%\frac{1}{nr(n-r)}\left(k(r-n)^2 + 2(r-k)(r-n)r + (n - 2r +k)r^2\right)=
%=& \frac{1}{nr(n-r)}\left[k(n-r)^2 + 2k(n-r)r + kr^2 + (n-2r)r^2 - 2(n-r)r^2 \right] \\
%=& \frac{1}{nr(n-r)}\left(kn^2 - nr^2\right)= \frac{kn - r^2}{r(n-r)} = \frac{(k/r)(n/r) - 1}{n/r - 1} = \frac{(k/r)n/r - 1}{n/r - 1} \\
 \frac{k}{r}\cdot \frac{n/r}{n/r - 1} - \frac{1}{n/r - 1},
%\end{align*}
\ee
i.e.,  the inner product is a function of $k/r$ and $n/r$.  In particular, it is  minimized when $k=0$, or equivalently    when $S\sim S'$.  
Lastly, we show that     $S\to  p_S$ is  an optimal vector coloring. First, note that $\chiv(K_{n:r})=n/r$. This follows by Theorem~\ref{thm:1walkreg} $(i)$ using the fact that $K_{n:r}$ is $\binom{n-r}{r}$-regular and  its least eigenvalue  is $-\binom{n-r-1}{r-1}$ (e.g.,~see \cite[Theorem 9.4.3]{GR}).
On the other hand, for    $S\sim S'$ it follows by \eqref{eq:canonicalvckneser} that  
 $\langle p_S, p_{S'}\rangle =-\frac{1}{n/r - 1}$.% and  thus
%where
%\[\alpha = \frac{[r]_q-[n]_q}{\sqrt{[n]_q[r]_q([n]_q-[r]_q)}} \quad \& \quad \beta = \frac{[r]_q}{\sqrt{[n]_q[r]_q([n]_q-[r]_q)}}.\]
%Here, $[k]_q = \frac{q^k - 1}{q-1} = \sum_{i=0}^{k-1} q^i$ is the number of lines contained in a $k$-dimensional subspace of $\mathbb{F}_q^n$.

The  vector coloring of  the $q$-Kneser graph $qK_{n:r}$ is defined analogously.  Specifically, set  
$$[k]_q:= \frac{q^k - 1}{q-1} = \sum_{i=0}^{k-1} q^i,$$ which is     the number of lines contained in a $k$-dimensional subspace of $\mathbb{F}_q^n$. Then, to an $r$-dimensional subspace $S$ of  $\mathbb{F}_q^n$ we assign the  unit vector $p_S\in \R^{[n]_q}$, with entries indexed by the lines of $\mathbb{F}_q^n$, given~by:
%Let $P$ be a matrix with rows indexed by the $r$-dimensional subspaces of $\mathbb{F}_q^n$ (i.e.~by the vertices of $qK_{n:r}$) and columns indexed by the lines (1-dimensional subspaces) of $\mathbb{F}_q^n$ such that
\begin{equation}\label{eq:leasteigkneser}
p_{S}(\ell)= 
\begin{cases}
\frac{[r]_q-[n]_q}{\sqrt{[n]_q[r]_q([n]_q-[r]_q)}},  & \text{ if } \ell \subseteq S, \\
\frac{[r]_q}{\sqrt{[n]_q[r]_q([n]_q-[r]_q)}}, & \text{ if } \ell \cap S = \{0\}.
\end{cases}
\end{equation}
Lastly, {note}  that $qK_{n:r}$ is $q^{r^2}{n-r \brack r}_q$ regular  and its least eigenvalue is  equal to $-q^{r(r-1)}{n-r-1 \brack r-1}_q$ (e.g.,~see~\cite{EKR}). Here  ${n \brack k}_q$ denotes  the Gaussian binomial coefficient which is equal to the number of $k$-dimensional subspaces of $\mathbb{F}_q^n$. As $qK_{n:r}$ is 1-walk-regular it follows by Theorem~\ref{thm:1walkreg} $(i)$ that 
$\chiv(qK_{n:r})= [n]_q/[r]_q$. To see that \eqref{eq:leasteigkneser} is an optimal vector coloring note that for two $r$-dimensional subspaces with a   $k$-dimensional intersection we have   
\be
\langle p_S, p_{S'}\rangle =%\frac{1}{nr(n-r)}\left(k(r-n)^2 + 2(r-k)(r-n)r + (n - 2r +k)r^2\right)=
%=& \frac{1}{nr(n-r)}\left[k(n-r)^2 + 2k(n-r)r + kr^2 + (n-2r)r^2 - 2(n-r)r^2 \right] \\
%=& \frac{1}{nr(n-r)}\left(kn^2 - nr^2\right)= \frac{kn - r^2}{r(n-r)} = \frac{(k/r)(n/r) - 1}{n/r - 1} = \frac{(k/r)n/r - 1}{n/r - 1} \\
 \frac{[k]_q}{[r]_q}\cdot \frac{[n]_q/[r]_q}{[n]_q/[r]_q - 1} - \frac{1}{[n]_q/[r]_q - 1}.
%\end{align*}
\ee
In particular, when $S\sim S'$ (i.e., $k=0$) we get that 
\[\langle p_S, p_{S'}\rangle=-\frac{1}{{[n]_q/[r]_q} - 1}.\]

\section{Computations}\label{appendix:comput}

To compute the  dimension of the span of the $\{p_e: e\in E(G)\}$, we may just   calculate  the rank of their Gram matrix. The Gram matrix of the $\{p_e: e\in E(G)\}$  is the matrix $M$ indexed by the edges of $G$ such that $M_{ef} = \tr(p_ep_f)$. Note that the value of this trace is equal to the sum of the entries of the entrywise product of $p_e$ and $p_f$, which is the usual inner product if we were to consider $p_e$ and $p_f$ as vectors. If $e = \{i,j\} \in E(G)$ and $f = \{\ell, k\} \in E(G)$, then
\begin{equation}\label{eqn:graments}
M_{ef} = \tr\left((p_ip_j^\trans+p_jp_i^\trans)(p_\ell p_k^\trans+p_kp_\ell^\trans)\right) = 2(p_j^\trans p_\ell \cdot p_k^\trans p_i + p_j^\trans p_k \cdot p_\ell^\trans p_i).
\end{equation}
Now, let $\tau$ be the least eigenvalue of $G$ and let $E_\tau$ be the projection onto its $\tau$-eigenspace. Then, recalling the definition of the canonical vector coloring, we have that $p_j^\trans p_\ell = (E_\tau)_{j\ell}$, and similarly for the other inner products appearing in~\eqref{eqn:graments}. Thus, to compute the entries of $M$, it suffices to compute the entries of $E_\tau$. Moreover, it suffices to compute a nonzero multiple of $E_\tau$ since scaling $E_\tau$ by $\gamma$ translates to scaling $M$ by $\gamma^2$, which does not affect its rank. We now describe how to do this under the assumption that $\tau$ is an integer.

Let $\lambda_1 \ge \ldots \ge \lambda_n = \tau$ be the eigenvalues of $G$ (including multiplicities) in decreasing order. Also, let $\phi$ be the characteristic polynomial of the adjacency matrix of $G$. Then $\phi$ is a monic polynomial with integer coefficients and
\[\phi(x) = \prod_{i = 1}^n (x - \lambda_i).\]
Also, define the polynomial $\phi_\tau$ as
\[\phi_\tau(x) = \prod_{\lambda_i \ne \tau} (x - \lambda_i).\]
If $d$ is the multiplicity of $\tau$ as an eigenvalue of $G$, then
\[\phi_\tau(x) = \frac{\phi(x)}{(x-\tau)^d}.\]
Since $(x - \tau)^d$ is a factor of $\phi(x)$, if $\tau$ is an integer, then $\phi_\tau$ is a monic polynomial with integer coefficients.

Now let $A$ be the adjacency matrix of $G$ and consider the matrix
\[\phi_\tau(A) = \prod_{\lambda_i \ne \tau} (A - \lambda_i I).\]
Note that all of the factors in the above product commute. If $v$ is an eigenvector of $A$ for an eigenvalue other than $\tau$, then it is easy to see that $\phi_\tau(A)v = 0$. On the other hand, if $v$ is a $\tau$-eigenvector of $A$, then
\[\phi_\tau(A)v = \left(\prod_{\lambda_i \ne \tau} (\tau - \lambda_i)\right)v \ne 0.\]
In other words, $\phi_\tau(A)$ is a nonzero multiple of $E_\tau$. Thus, for a 1-walk-regular graph $G$ with adjacency matrix $A$ and integer least eigenvalue $\tau$, we have the following algorithm for determining $\dim\left(\spn\{p_e: e\in E(G)\}\right)$:
\begin{enumerate}
\item Compute the characteristic polynomial $\phi$ of $A$.
\item Compute $\phi_\tau$ by repeatedly dividing $\phi(x)$ by $(x-\tau)$.
\item Compute $\phi_\tau(A)$, which is a multiple of $E_\tau$.
\item Use $\phi_\tau(A)$ and~\eqref{eqn:graments} to compute the Gram matrix, $M$, of the $p_e$.
\item Compute the rank of $M$ which is equal to $\dim\left(\spn\{p_e: e\in E(G)\}\right)$.
\end{enumerate}
Importantly, each of these steps can be done efficiently \emph{and} exactly with integer arithmetic~\cite{charpoly}. By Lemma~\ref{cor:coretest}, this algorithm allows us to determine a sufficient condition for showing  a 2-walk-regular graph is a core.  
%then we can conclude that $G$ is a core by Corollary~\ref{cor:coretest}. 
We note that requiring the least eigenvalue of $G$ to be integer does not seem to be very restrictive in practice.

It appears that the significant majority of the computational time is spent on determining the rank of the Gram matrix of the $p_e$ matrices. Based on our experience, the runtime of the algorithm appears to be roughly quadratic in the number of edges of the graph. At 20,000 edges it takes about 15 minutes for the algorithm to run in Sage~\cite{sage} on our personal computers. Note that computationally testing whether a graph of this size is a core is essentially~impossible.

\begin{table}[ht!]\label{tab:comp}
\caption{Data for Strongly Regular Graphs}
\centering
\[
\begin{array}{lrrrr}
\text{Param. Set} & \text{Total Num.} & \text{Num. Tight} & \text{Num. Loose}\\\hline
(9, 4, 1, 2) &1& 0& 1\\
(10, 3, 0, 1) &1& 1& 0\\
(10, 6, 3, 4) &1& 1& 0\\
(15, 6, 1, 3)& 1& 1& 0\\
(15, 8, 4, 4) &1& 0& 1\\
(16, 5, 0, 2) &1 &1 &0\\
(16, 10, 6, 6) &1 &1& 0\\
(16, 6, 2, 2) &2 &0 &2\\
(16, 9, 4, 6) &2 &1 &1\\
(21, 10, 3, 6) &1 &1 &0\\
(21, 10, 5, 4) &1 &0& 1\\
(25, 8, 3, 2) &1 &0 &1\\
(25, 16, 9, 12)& 1& 0& 1\\
(25, 12, 5, 6) &15 &13 &2\\
(26, 10, 3, 4) &10& 9& 1\\
(26, 15, 8, 9) &10& 9& 1\\
(27, 10, 1, 5) &1& 1& 0\\
(27, 16, 10, 8) &1& 0& 1\\
(28, 12, 6, 4) &4& 0& 4\\
(28, 15, 6, 10)& 4& 4& 0\\
(35, 16, 6, 8)& 3854& 2789 &1065\\
(35, 18, 9, 9)& 3854& 2175 &1679\\
(36, 10, 4, 2)& 1& 0& 1\\
(36, 25, 16, 20)& 1& 0 &1\\
(36, 14, 4, 6) &180 &175& 5\\
(36, 21, 12, 12)& 180& 135& 45\\
(36, 14, 7, 4) &1& 0& 1\\
(36, 21, 10, 15)& 1& 1& 0\\
(36,15,6,6) & 32548 & 24022 & 8526 \\
(36,20,10,12) & 32548 & 32536 & 12 \\
(40, 12, 2, 4) &28 &16 &12\\
(40, 27, 18, 18)& 28& 17& 11\\
(45, 12, 3, 3) &78 &0 &78\\
(45, 32, 22, 24)& 78& 77& 1\\
(45, 16, 8, 4) & 1 & 0 & 1 \\
(45, 28, 15, 21) & 1 & 1 & 0 \\
(49,12,5,2) & 1 & 0 & 1 \\
(49,36,25,30) & 1 & 0 & 1 \\
(50,7,0,1) & 1 & 0 & 1 \\
(50,42,35,36) & 1 & 1 & 0 \\
(50,21,8,9) & 18 & 18 & 0 \\
(50, 28, 15, 16) & 18 & 17 & 1 \\
(64, 18, 2, 6) &167 &145 &22\\
(64, 45, 32, 30)& 167& 0& 167
\end{array}
\]
\label{srgtable}
\end{table}
\afterpage{\clearpage}

Furthermore, we note that in~\cite{UVC1} we presented another algorithm for determining whether a 1-walk-regular graph $G$ is uniquely vector colorable. This algorithm was based on solving a system of $|V(G)|^2$ linear equations in $|E(\overline{G})|$ variables. This is somewhat complementary to the algorithm given here, whose runtime depends on $|E(G)|$. In practice, the algorithm introduced above is much faster than the algorithm given in the previous work.% But perhaps the earlier algorithm could be faster on some large dense graphs.

We now apply this algorithm to Ted Spence's list of strongly regular graphs  available online~\cite{Spence}. In this case, we can actually compute the entries of the Gram matrix of the matrices $\{p_e : e\in E(G)\}$ directly from the parameters of the strongly regular graph $G$, which saves us some work. Furthermore, $G$ has integral eigenvalues unless it is a conference graph, i.e., has parameters $(4t+1,2t,t-1,t)$ for some integer $t$. Even in this case, $G$ still has integral eigenvalues unless $4t-1$ is not a square. Thus, considering only graphs with integral least eigenvalue  is not a significant restriction for SRGs. After eliminating the graphs with non-integral eigenvalues, we were left with 73816 graphs. For each of these, we computed the Gram matrix of the matrices  $\{p_e : e\in E(G)\}$ and then calculated its rank and compared the result to $d+1 \choose 2$, where $d$ is the multiplicity of the least eigenvalue of $G$. 
We call a graph \emph{tight} if the rank  is $d+1 \choose 2$, and \emph{loose} otherwise. For each tight graph $G$, this calculation certifies that $G$ is a core. The results of these calculations are given in Table~\ref{srgtable}.

%It turns out that if $G$ is strongly  regular,  we can compute the entries of the Gram matrix of the matrices  $\{p_e : e\in E(G)\}$ directly from the parameters of~$G$. Furthermore, $G$ has integral eigenvalues unless it is a conference graph, i.e., has parameters $(4t+1,2t,t-1,t)$ for some integer $t$. Even in this case, $G$  still has integral eigenvalues unless $4t-1$ is not a square. Thus, considering only graphs with integral least eigenvalue  is not a significant restriction for SRGs.

We have made some interesting observations from the data we have collected so far. There are 8526 SRGs with parameters $(36,15,6,6)$ that are not tight, but we have verified that these are all cores. The 10 graphs with parameter set $(26,10,3,4)$ are the Paulus graphs and the  10 graphs with parameter set $(26,15,8,9)$ are their complements. For each set, exactly one graph is not tight. These two graphs form a complementary pair. They correspond to the Paulus graph with the largest automorphism group (the size of this group is 120, the next largest automorphism group has size 39, the remaining Paulus graphs have less than 10 automorphisms). %This is as expected. Our test finds graphs that have unique vector colourings. The more automorphisms a graph has, the greater chance that one of them yields a distinct vector colouring.
The 180 graphs with parameter set $(36,14,4,6)$ correspond to a class of real symmetric Hadamard matrices with constant diagonal. All but 5 of these graphs are tight. Even more striking are the 32548 graphs with parameters $(36,20,10,12)$, all but 12 of which are tight.

%\bibliographystyle{plainurl}
%
%\bibliography{main.bib}

\end{document}